\documentclass[11pt]{amsart}
\usepackage{geometry}                
\usepackage{amssymb,amsmath,amsfonts,amscd,euscript}
\numberwithin{equation}{section}
\newtheorem{thm}{Theorem}[section]
\newtheorem{prop}[thm]{Proposition}
\newtheorem{lem}[thm]{Lemma}
\newtheorem{cor}[thm]{Corollary}
\theoremstyle{remark}
\newtheorem{rem}[thm]{Remark}

\newtheorem{dfn}[thm]{Definition}

\newcommand{\bc}{{\mathbb C}}
\newcommand{\pc}{{\mathbb P}}

\newcommand{\ui}{{\underline{i}}}
\newcommand{\uj}{{\underline{j}}}
\newcommand{\g}{{\mathfrak{g}}}
\newcommand{\n}{{\mathfrak{u}}}

\newcommand{\lfr}{{\mathfrak{l}}}

\newcommand{\fn}{{\mathfrak{n}}}

\newcommand{\nn}{{\mathfrak{n}^-}}

\newcommand{\p}{{\mathfrak{p}}}

\newcommand{\nc}{\newcommand}
\nc{\cO}{\mathcal O}
\nc{\cF}{\mathcal F}
\nc{\cL}{\mathcal L}
\nc{\msl}{\mathfrak{sl}}
\nc{\mgl}{\mathfrak{gl}}
\nc{\U}{\mathrm U}
\nc{\bH}{\EuScript H}
\nc{\Res}{\mathrm{Res\ }}

\newcommand{\bC}{{\mathbb C}}
\newcommand{\bR}{{\mathbb R}}
\newcommand{\bZ}{{\mathbb Z}}

\newcommand{\bP}{{\mathbb P}}
\newcommand{\bG}{{\mathbb G}}
\newcommand{\fp}{{\mathfrak p}}

\newcommand{\fb}{{\mathfrak b}}

\newcommand{\fh}{{\mathfrak h}}

\newcommand{\Fl}{\EuScript{F}}
\newcommand{\SF}{Sp\EuScript{F}}

\newcommand{\spa}{\mathrm{span}}

\newcommand{\bU}{{\bf U}}
\newcommand{\bV}{{\bf V}}
\newcommand{\bd}{{\bf d}}
\newcommand{\bs}{{\bf S}}
\newcommand{\bp}{{\bf p}}
\newcommand{\bS}{{\bf S}}
\nc{\ch}{\mathrm{ch}}
\nc{\la}{\lambda}
\nc{\msp}{\mathfrak{sp}}
\nc{\cd}{\cdots}
\nc{\hk}{\hookrightarrow}
\nc{\T}{\otimes}
\nc{\al}{\alpha}
\nc{\om}{\Omega}

\title[Symplectic degenerate flag varieties]
{Symplectic degenerate flag varieties}

\author{Evgeny Feigin, Michael Finkelberg and Peter Littelmann}
\address{Evgeny Feigin:\newline
Department of Mathematics, National Research University Higher School of Economics,\newline
20 Myasnitskaya st, 101000, Moscow, Russia\newline
{\it and }\newline
Tamm Theory Division,
Lebedev Physics Institute
}
\email{evgfeig@gmail.com}
\address{Michael Finkelberg:\newline
IMU, IITP, and National Research University Higher School of Economics,\newline
Department of Mathematics, 20 Myasnitskaya st, 101000, Moscow, Russia}
\email{fnklberg@gmail.com}
\address{Peter Littelmann:\newline
Mathematisches Institut, Universit\"at zu K\"oln,\newline
Weyertal 86-90, D-50931 K\"oln,Germany
}
\email{littelma@math.uni-koeln.de}

\begin{document}
\begin{abstract}
Let $\SF^a_\lambda$ be the degenerate symplectic flag variety. These are projective singular irreducible
$\bG_a^M$ degenerations of the classical flag varieties for symplectic group $Sp_{2n}$. 
We give an explicit
construction for the varieties $\SF^a_\lambda$ and construct their desingularizations,
similar to the Bott-Samelson resolutions in the classical case. We prove that $\SF^a_\la$ 
are normal locally complete intersections with terminal and rational singularities.
We also show
that these varieties are Frobenius split. Using the above mentioned results, we 
prove an analogue of the Borel-Weil-Bott theorem and obtain a $q$-character formula
for the characters of irreducible $Sp_{2n}$-modules via the Atiyah-Bott-Lefschetz fixed
points formula.  
\end{abstract}

\maketitle
\vskip 5pt\noindent
\section{Introduction}
Let $G$ be a complex simple algebraic group and $\g$ be the corresponding Lie algebra. 
Let $\g=\fb\oplus\fn^-$ be the Cartan decomposition, where $\fn^-$ is a nilpotent subalgebra 
in $\g$ and $\fb$ is the opposite Borel subalgebra. Let $B$ and $N^-$ 
be the subgroups in $G$ corresponding to $\fb$ and $\fn^-$.
The Lie algebra $\g$ has a degeneration $\g^a=\fb\oplus (\fn^-)^a$, where $(\fn^-)^a$ is 
the abelian Lie algebra with the underlying vector space $\fn^-$ (see \cite{Fe1}, \cite{Fe2}).
Here $(\fn^-)^a$ is an abelian ideal in $\g^a$ and $\fb$ acts on $(\fn^-)^a$ via the adjoint action on
the quotient $(\fn^-)^a\simeq \g/\fb$.
The corresponding Lie group $G^a$ is isomorphic to the semi-direct product $B\ltimes \bG_a^M$,
where $\bG_a$ is the additive group of the field and $M=\dim \fn^-$.  

For a dominant integral 
weight $\la$ let $V_\la$ be the corresponding irreducible highest weight $\g$-module.
Let $v_\la\in V_\la$ be a highest weight vector, recall that $V_\la=U(\fn^-)v_\la$.
The (generalized) flag varieties $G/P$ ($P$ being a parabolic subgroup of $G$) are known to be embedded 
into the projective spaces $\bP(V_\la)$ with $\la$ chosen in such a way that $P$ 
is the stabilizer of $\bC v_\la$. Explicitly, the image is given by the $G$-orbit
through the highest weight line $\bC v_\la\in\bP(V_\la)$. We denote the corresponding orbit by
$\Fl_\la$. 

Let $V_\la^a$ be the degeneration of $V_\la$ into a $\g^a$-module. More precisely,
$V_\la^a$ is the associated graded space with respect to the PBW filtration on $V_\la$
(see e.g. \cite{FFoL1}, \cite{FFoL2}). 
Denote by the same symbol $v_\la$ its image in $V_\la^a$, then
$V_\la^a=S^\bullet(\fn^-) v_\la$.
The degenerate flag variety $\Fl^a_\la$ is defined as the closure of the orbit $G^a (\bC v_\la)$ inside
$\bP(V_\la^a)$. 
In contrast with the classical situation, the orbit itself is not closed, it is only an open cell inside $\Fl^a_\la$. 

The varieties $\Fl^a_\la$ for $\g=\msl_n$ were studied in \cite{Fe1}, \cite{Fe2}, \cite{FF}.
It was shown that in this case the $\Fl^a_\la$ are singular projective algebraic varieties, which
are flat degenerations of the classical $\Fl_\la$. The varieties $\Fl^a_\la$ enjoy several
nice properties: as in the classical case, 
$\Fl^a_\la$ depends only on the class of regularity of $\la$;
they are irreducible normal Frobenius split locally complete intersections; 
they have a nice crepant desingularization isomorphic to a tower of  successive $\bP^1$ fibrations;  
the singularities of $\Fl^a_\la$ are rational. In addition, the analogue of the classical Borel-Weil-Bott
theorem still holds in the degenerate case.

In this paper we study the case of the symplectic group $G=Sp_{2n}$. Since we often use the connection 
between $SL_{2n}$ flag varieties and the $Sp_{2n}$ flag varieties, we use the notation $\SF^a_\la$
to distinguish it from the type ${\tt A}$ case. 
In the introduction we describe only the case of the complete degenerate symplectic flag variety,  which we
denote by $\SF^a_{2n}$. (In this case the highest weight $\la$ has to be regular, but, as in the classical case,
the orbit closure does not depend on a regular $\la$). However, in the main body of the paper we work 
out the general case of degenerate parabolic (partial) flag varieties as well.  

Let $W$ be a $2n$-dimensional vector space with
a basis $w_1,\dots,w_{2n}$. Let us equip $W$ with a symplectic form which pairs non-trivially  
$w_i$ and $w_{2n+1-i}$. We also denote by $pr_k:W\to W$, $k=1,\dots,2n$ the projection operators along
$w_k$ to the span of all the rest basis vectors.  Our first theorem is as follows:
\begin{thm}
$\SF^a_{2n}$ can be realized inside the product $\prod_{k=1}^n Gr_k(W)$ of Grassmannians as a subvariety of
collections $(V_1,\dots,V_n)$, $V_k\in Gr_k(W)$ satisfying the following conditions:
\[
pr_{k+1} V_k\subset V_{k+1},\ k=1,\dots,n-1,\ V_n \text{ is Lagrangian}. 
\]
These varieties are flat degenerations of the classical flag varieties $Sp_{2n}/B$.  
\end{thm}
      
Our next task is to study the singularities of $\SF^a_{2n}$. As in the type ${\tt A}$ case, we construct the 
desingularization $SpR_{2n}$ of $\SF^a_{2n}$. We prove that $SpR_{2n}$ is a Bott-Samelson type variety, i.e.
it is isomorphic to a tower of successive $\bP^1$-fibrations. Using this desingularization, we prove our next 
theorem:  
\begin{thm}
The varieties $\SF^a_{2n}$ are normal locally complete intersections and thus Cohen-Macaulay and Gorenstein.
The singularities of $\SF^a_{2n}$ are terminal and rational. 
The varieties $\SF^a_{2n}$ are Frobenius split over $\overline{\mathbb F}_p$ 
for all primes $p$.
\end{thm}           
We note that an important difference with the type ${\tt A}$ case is that the 
resolution of singularities $SpR_{2n}\to \SF^a_{2n}$ is no longer crepant,
though it seems likely the resolution is still semismall.

Finally, we make a connection between the geometry of $\SF^a_{2n}$ and the representations $V_\la^a$.
Namely, we prove an analogue of the  Borel-Weil-Bott theorem (see e.g. \cite{K}). Let us denote by $\imath_\la$ the natural 
map $\SF^a_{2n}\to \bP(V_\la^a)$. Then we have:
\begin{thm}
$H^0(\SF^a_{2n},\imath_\la^* \cO(1))^*=V_\la^a$, and
$H^k(\SF^a_{2n},\imath_\la^* \cO(1))^*=0$ for $k>0$.
\end{thm}
Using the rationality of the singularities of $\SF^a_{2n}$ and the Atiyah-Bott-Lefschetz fixed points formula for
$SpR_{2n}$ we derive a formula for the $q$-character of $V_\la^a$ (see \cite{FFoL2} for the combinatorial formula).

Our paper is organized as follows: \\
In Section \ref{Definitions} we introduce notation, recall definitions and collect main results
to be used in the main body of the paper.\\
In Section \ref{Filtrations} we establish a connection between the PBW filtrations on $\msp_{2n}$-modules 
and on $\msl_{2n}$-modules.\\
In Section \ref{Explicit} we derive an explicit description for the symplectic Grassmannians and, more generally, 
for the symplectic degenerate flag varieties. We also prove that $\SF_\la^a$ are flat degeneration of their classical
analogues $\SF_\la$.\\
In Section \ref{Resolution} we construct a resolution of singularities of $\SF^a_\la$.\\
In Section \ref{Normal_l.c.i.} we prove that all symplectic degenerate flag varieties are normal locally 
complete intersections and thus Cohen-Macaulay and Gorenstein.\\
In Section \ref{Singularities} we show that the singularities of $\SF^a_\la$ are terminal and thus canonical and
rational.\\
In Section \ref{Applications} we prove that the varieties $\SF^a_\la$ are Frobenius split. This allows
to prove the degenerate analogue of the Borel-Weil-Bott theorem. We also derive an Atiyah-Bott-Lefschetz type
formula for the graded characters of $V_\la^a$.

\section{Degenerations: definitions and the type ${\tt A}$ case}\label{Definitions}
In this section we fix the notation and recall the main results on the degenerate representations and  
degenerate flag varieties. We also collect important constructions and theorems  from algebraic geometry, 
which we use  in the paper.

\subsection{Notation.}
Let $\g$ be a simple Lie algebra. We fix a Cartan decomposition $\g=\fn\oplus\fh\oplus\fn^-$. 
Let $\fb=\fn\oplus\fh$ be the corresponding Borel subalgebra. We denote by $\al_i$, $i=1,\dots,\mathrm{rk}(\g)$
the simple roots of $\g$ and by $\omega_i$ the corresponding fundamental weights. Let $\Phi^+$ be
the set of positive roots for $\g$. For a root $\al$ we often write $\al>0$ instead of $\al\in \Phi^+$.
In particular, for $\g=\msl_n$ the simple roots are $\al_1,\dots,\al_{n-1}$ and all positive roots 
are of the form
\begin{equation}\label{prsln}
\al_{i,j}=\al_i+\al_{i+1}+\dots + \al_j, 1\le i\le j\le n-1.
\end{equation}
For $\g=\msp_{2n}$ the simple roots are $\al_1,\dots,\al_n$ and  all positive roots are of the form 
\begin{gather} \label{ij}
\al_{i,j}=\al_i+\al_{i+1}+\dots +\al_j,\ 1\le i\le j\le n,\\
\label{ibarj}
\al_{i,j}=\al_i+\al_{i+1}+\dots +\al_n +\al_{n-1}+ \dots +\al_{2n-j},\ 1\le i\le n <j, i+j\le 2n.   
\end{gather}
For a positive 
root $\al$ we fix a weight $-\al$ element $f_\al\in \fn^-$ and a weight $\al$ element $e_\al\in \fn$.
Then the $f_\al$, $\al\in \Phi^+$ form a basis in $\fn^-$ and so do the root vectors $e_\al$ in $\fn$.
We use the shorthand notations $f_{i,j}=f_{\al_{i,j}}$.

Let $P^+ = \sum_{i=1}^{\mathrm{rk}(\g)} \bZ_{\ge 0}\omega_i$ be the submonoid of the weight lattice for $\g$
generated by the fundamental weights, the elements of $P^+$ are the dominant integral weights. 
For $\la\in P^+$ we denote by $V_\la$ the $\g$-module of highest weight $\la$. 
We also fix a highest weight vector $v_\la\in V_\la$. In particular, one has
\[
\fn v_\la=0,\ hv_\la=\la(h)v_\la, h\in\fh,\ V_\la=U(\fn^-) v_\la.
\] 

Let $G$ be a simple complex algebraic group with the Lie algebra $\g$. We denote by $B,N,T,N^-$ the 
subgroups in $G$ corresponding to the Lie subalgebras $\fb,\fn,\fh,\fn^-$. Let $P$ be a parabolic subgroup in $G$,
the corresponding generalized flag variety is the quotient $G/P$. These varieties can be also
realized as follows: for a dominant integral weight $\la$ the group $G$ acts naturally on the 
projective space $\bP(V_\la)$. Assume that $(\la,\omega_i) = 0$ if and only if $f_{\al_i}$ 
belongs to the Lie algebra $\fp$ of $P$. Then $G/P$ is isomorphic to the $G$-orbit 
$G (\bC v_\la)\subset \bP(V_\la)$. 

In what follows we denote the orbit $G (\bC v_\la)$
by $\Fl_\la$. Note that $\Fl_\la$ depends only on the class of regularity of $\la$, i.e. 
$\Fl_\la\simeq \Fl_\mu$ if and only if the supports of $\la$ and $\mu$ coincide, i.e. 
$(\la,\omega_i)\ne 0$ if and only if $(\mu,\omega_i)\ne 0$. 

For example, for $G=SL_n$ the Grassmann variety $Gr_d(n)$ of $d$-dimensional
subspaces in $n$-dimensional space can be realized either as the quotient of $SL_n$ by a maximal parabolic
subgroup or as an orbit in the projective space of the fundamental 
representation $V_{\omega_d}$.
In general, all the quotients $SL_n/P$, where $P\supseteq B$ is a parabolic subgroup,  are partial flag varieties:  
let $1\le d_1<\dots< d_k\le n$ be the integers such that $f_{\al_d}\in\fp$ if 
and only if $d=d_i$ for some $i$.
Then $G/P$ is known to coincide with the variety of collections $(V_i)_{i=1}^k$ of subspaces $V_i$ in an 
$n$-dimensional space such that $\dim V_i=d_i$ and $V_i\subset V_{i+1}$. We note also that 
the same partial flag variety sits inside 
$\bP(V_{\omega_{d_1}+\dots +\omega_{d_k}})$ as the orbit of the highest weight line.
We denote this variety by $\Fl_\bd$, $\bd=(d_1,\dots,d_k)$.  

Another important example for us is $G=Sp_{2n}$. The symplectic Grassmannians $SpGr_d(2n)$ (quotients of
$Sp_{2n}$ by maximal parabolic subgroups) are known to coincide with the variety of isotropic
$d$-dimensional subspaces in a $2n$-dimensional vector space equipped with a non-degenerate symplectic form.
In general, as in the $SL_{2n}$ case,
let $1\le d_1<\dots< d_k\le n$ be the numbers such that $f_{\al_d}\in\fp$ if and only if $d=d_i$ for some $i$
($\fp$ is the Lie algebra of a parabolic subgroup $P\subset Sp_{2n}$).
Then $Sp_{2n}/P$ is known to coincide with the variety of collections $(V_i)_{i=1}^n$ of subspaces 
$V_i\in SpGr_{d_i}(2n)$ such that $V_i\subset V_{i+1}$. In addition,  
the same partial flag variety sits inside 
$\bP(V_{\omega_{d_1}+\dots +\omega_{d_k}})$ as the orbit of the highest weight line.     
We denote this variety by $Sp\Fl_\bd$, $\bd=(d_1,\dots,d_k)$.

\subsection{Algebraic geometry.} 
We  recall the definition of a Frobenius split variety (see \cite{BK}). Let  $X$ be an algebraic 
variety over an algebraically closed field of characteristics $p>0$.
Let $F:X\to X$ be the Frobenius morphism, i.e. an identity map on the underlying space
$X$ and the $p$-th power map on the space of functions.
Then $X$ is called Frobenius split if there exists a projection $F_*\cO_X\to \cO_X$ such that
the composition $\cO_X\to  F_*\cO_X\to \cO_X$ is the identity map.
The Frobenius split varieties enjoy the following important property (see. e.g.
Proposition $1$ of \cite{MR}):
\begin{prop}
Let $X$ be a Frobenius split projective variety with a line bundle $\cL$ such that for some $i$ and
all large enough $m$  $H^i(X,\cL^m)=0$. Then $H^i(X,\cL)=0$.
\end{prop}

The following two statements are proved in \cite{MR}, Proposition $4$ and Proposition~$8$:
\begin{prop}
Let $f:Z\to X$ be a proper morphism of algebraic varieties such that $f_*\cO_Z=\cO_X$.
If $Z$ is Frobenius split, then $X$ is also Frobenius split.
\end{prop}

\begin{thm}\label{MR}
Let $Z$ be a smooth projective variety of dimension $M$ and let $Z_1, \dots,Z_M$ be codimension one
subvarieties satisfying the following conditions:
\begin{enumerate}
\item $\forall\,I\subset\{1,\dots,M\}$: the intersection $\cap_{i\in I} Z_i$
is smooth of codimension $\#I$.
\item There exists a global section $s$ of the anti-canonical bundle $\omega^{-1}$ on $Z$ such that the 
zero divisor of $s$ equals $\cO(Z_1+\dots + Z_M+D)$
for some effective divisor $D$ with $\cap_{i=1}^M Z_i\notin supp D$.
\end{enumerate}
Then $Z$ is Frobenius split and for any subset $I\subset \{1,\dots,M\}$ the intersection
$Z_I=\cap_{i\in I} Z_i$ is Frobenius split as well.
\end{thm}

Now we recall some results on the singularities of algebraic varieties. 
Let $X$ be a projective algebraic variety. If $X$ is smooth, then $\omega_X$ denotes its canonical 
line bundle of top degree forms on $X$.
For a singular $X$ one can define a dualizing complex $DX$, 
which for Cohen-Macaulay varieties reduces
to a (cohomologically shifted) sheaf. 
If $X$ is  Gorenstein, then this sheaf is a line bundle denoted by $\omega_X$. 
For example if $X$ is a locally complete intersection, then it is known that $X$ is Gorenstein and 
thus $\omega_X$ is a line bundle.

Let $\pi:Y\to X$ be a resolution of singularities of $X$. The singularities of $X$ are called rational if
$\pi_* \cO_Y=\cO_X$ and $R^k\pi_*\cO_Y=0$ for $k>0$ (if this holds for some desingularization, then
it holds for all). An important property is that in this case for any line bundle $\cL$ on $X$
and any $k$ one has $H^k(X,\cL)=H^k(Y,\pi^*\cL)$. 

Let $Z\subset Y$ be an irreducible divisor. Then $Z$ is said to be exceptional if $\dim \pi(Z)<\dim Z$. 
Assume that 
$$\omega_Y=\pi^*\omega_X\T\bigotimes_{i=1}^N \cO(a_iZ_i),$$ 
where $Z_1,\dots,Z_N$ is the set of all
exceptional divisors in $Y$. Then the singularities of $X$ are called canonical, if $a_i\ge 0$ for all $i$.
If moreover all $a_i$ are positive then the singularities are called terminal. 
We will use the following theorem from \cite{E}, Theorem 1 (see also \cite{F}, $1.3$):
\begin{thm}
Let $X$ be a Gorenstein variety and $f:Y\to X$ be a desingularization such that the natural inclusion
$f_*\omega_Y\hk \omega_X$ is an isomorphism. Then $X$ has rational singularities.
\end{thm}

\subsection{Degenerate representations.}
Consider the PBW filtration $U(\fn^-)_s$ on the universal enveloping algebra of $U(\fn^-)$:
\[
U(\fn^-)_s=\spa (x_1\dots x_l: \ x_i\in\fn^-, l\le s).
\]
The associated graded algebra is isomorphic to the symmetric algebra $S^\bullet(\fn^-)$.
For any  $\la\in P^+$ consider the induced PBW filtration $F_s=U(\fn^-)_s v_\la$ on the space $V_\la$.
We denote the associated graded space by $V^a_\la$, so $V_\la^a$ carries additional grading:
\[
V_\la=\bigoplus_{s\ge 0} V_\la(s)=\bigoplus_{s\ge 0} F_s/F_{s-1}.  
\]
Note that the operators from $\fn^-$ acting on $V_\la$ induce an action of the abelian
algebra $(\fn^-)^a$ on $V^a_\la$, where $(\fn^-)^a$ is isomorphic to $\fn^-$ as vector spaces.  
Thus $V_\la^a=S^\bullet(\fn^-)v_\la$, where (slightly abusing notations) we denote by 
$v_\la\in V_\la^a$ the image of the highest weight vector in $V_\la$.
More concretely,
$V_\la^a\simeq \bC[f_\al]_{\al>0}/I_\la$, where $I_\la$ is an ideal.

We note that $(\fn^-)^a$ comes equipped with the natural structure of $\fb$-module. In fact, we have
isomorphism of vector spaces $(\fn^-)^a\simeq \g/\fb$ and $\fb$ acts on the right hand side  via the
adjoint action. This gives the structure of $\fb$-module on $S^\bullet(\fn^-)$.
We denote this $\fb$-action by $\circ$. By construction,
the ideals $I_\la$ are $\fb$-invariant and thus $\fb$ acts on all modules $V_\la^a$. 
It is convenient to combine the actions of $(\fn^-)^a$ and $\fb$ together.  
Namely, 
consider the Lie algebra 
$\g^a\simeq \fb\oplus (\fn^-)^a$, where $(\fn^-)^a$ is an abelian ideal and the Borel subalgebra $\fb$
acts on $(\fn^-)^a$ via the induced adjoint action on the quotient space $(\fn^-)^a\simeq \g/\fb$.
Then each $V_\la^a$ carries the structure of a $\g^a$-module (see \cite{Fe1}, \cite{Fe2}).
In what follows we will need explicit description of the ideals $I_\la$  for $\g=\msl_n$ (see \cite{FFoL1}),
and $\g=\msp_{2n}$ (see \cite{FFoL2}).

\begin{thm}\label{genrel}
Let $\g=\msl_n$, $\la=\sum_{i=1}^{n-1} m_i\omega_i$. Then 
\[
I_\la=S(\fn^-) \left( U(\fb)\circ f_{\al_{i,j}}^{m_i+\ldots+m_j+1}\mid 1\le i\le j \le n-1\right). 
\]
Let $\g=\msp_{2n}$, $\la=\sum_{i=1}^{n} m_i\omega_i$. Then 
$$
I_\lambda=S(\fn^-)\left(
\begin{array}{l}
U(\fb)\circ f_{\alpha_{i,j}}^{{m_i+\ldots+m_j+1}}\mid 1\le i \le j< n , \\
U(\fb)\circ f_{\alpha_{i,2n-i}}^{{m_i+\ldots+m_n+1}}\mid 1\le i\le n.\\
\end{array}
\right)
$$
\end{thm}

We will also use  monomial  bases in $V_\la^a$ (see \cite{FFoL1}, \cite{FFoL2}, \cite{V}).   
Namely, given an element ${\bf s}=(s_\al)_{\al>0}$, $s_\al\in Z_{\ge 0}$, we define 
$f^{\bf s}=\prod_{\al>0} f_\al^{s_\al}$. 
We need the definition of a Dyck path inside the set of positive roots for $\msl_n$ and $\msp_{2n}$. 

\begin{dfn}
We call a sequence $\bp=(\beta(0), \beta(1),\dots, \beta(k))$ 
of positive roots of $\msl_n$ a Dyck path if
$\beta(0)=\al_i$, $\beta(k)=\al_j$ for some $1\le i \le j\le n-1$ and
$$
\text{if } \beta(s)=\al_{p,q} \qquad \text{ then } \qquad \beta(s+1)=\al_{p,q+1} \text{ or }
\beta(s+1)=\al_{p+1,q}.
$$

We call a sequence $\bp=(\beta(0), \beta(1),\dots, \beta(k))$
of positive roots of $\msp_{2n}$ a (symplectic) Dyck path if
$\beta(0)=\al_i$, $\beta(k)=\al_j$ or $\beta(k)=\al_{j,2n-j}$ for some $1\le i \le j\le n-1$ and
$$
\text{if } \beta(s)=\al_{p,q} \qquad \text{ then } \qquad \beta(s+1)=\al_{p,q+1} \text{ or }
\beta(s+1)=\al_{p+1,q}.
$$
\end{dfn}

For a dominant $\msl_n$-weight $\la=\sum_{i=1}^{n-1} m_i\omega_i$ 
let $P_{\msl_n}(\la)\subset \bR^{\frac{1}{2}n(n-1)}_{\ge 0}$ be the polytope consisting of collections
$(r_\al)_{\al\in \Phi^+_{\msl_n}}$ such that for any $\msl_n$ Dyck path $\bp$ with
$\beta(0)=\al_i$, $\beta(k)=\al_j$ one has
\[
r_{\beta(0)} + \dots + r_{\beta(k)}\le m_i + \dots + m_j.
\]
Let $S_{\msl_n}(\la)$ be the set of integral points in $P_{\msl_n}(\la)$.

Similarly, for a dominant $\msp_{2n}$-weight $\la=\sum_{i=1}^{n} m_i\omega_i$ 
let $P(\la)\subset \bR^{n^2}_{\ge 0}$ be the polytope consisting of collections
$(r_ \al)_{\al\in \Phi^+_{\msp_{2n}}}$ such that for any $\msp_{2n}$ Dyck path $\bp$ with
$\beta(0)=\al_i$, $\beta(k)=\al_j$ one has
\[
r_{\beta(0)} + \dots + r_{\beta(k)}\le m_i + \dots + m_j
\]
and for any $\msp_{2n}$ Dyck path $\bp$ with $\beta(0)=\al_i$, $\beta(k)=\al_{j,2n-j}$ one has
\[
r_{\beta(0)} + \dots + r_{\beta(k)}\le m_i + \dots + m_n.
\]
Let $S_{\msp_{2n}}(\la)$ be the set of integral points in $P_{\msp_{2n}}(\la)$.

\begin{thm}\label{basis}
For a dominant $\msl_n$-weight $\la$ the monomials $f^{\bf s} v_\la$, ${\bf s}\in S_{\msl_n}(\la)$ form
a basis of $V_\la^a$. 
For a dominant $\msp_{2n}$-weight $\la$ the monomials $f^{\bf s} v_\la$, ${\bf s}\in S_{\msp_{2n}}(\la)$ form
a basis of $V_\la^a$. 
\end{thm}

Finally, for $\g=\msl_n$ and $\g=\msp_{2n}$ the representations $V_\la^a$ have the following important
property (similar to the classical situation): for two dominant weights $\la$ and $\mu$ the representation
$V^a_{\la+\mu}$ is embedded into the tensor product $V^a_\la\T V^a_\mu$ as the highest component, i.e.
there exists a homomorphism of $\g^a$-modules: 
\begin{equation}\label{Cartan}
V^a_{\la+\mu}\hk V^a_\la\T V^a_\mu,\ v_{\la+\mu}\mapsto v_\la\T v_\mu. 
\end{equation}

\subsection{Degenerate flag varieties.}
Here we follow \cite{Fe1}, \cite{Fe2} and \cite{FF}. 
Let $G^a$ be the Lie group of the Lie algebra $\g^a$. Explicitly, $G^a$ is a semi-direct product of the Borel 
subgoup $B$ and the normal abelian subgroup $(N^-)^a$, which is isomorphic to the product of $\dim \fn^-$
copies of the additive group of the field $\bG_a$. 
Let $\la$ be a dominant integral weight. Then $G^a$ naturally acts on $\bP(V_\la^a)$. By definition,
the degenerate flag variety $\Fl^a_\la$ is the orbit closure of the highest weight line, i.e.
\[
\Fl^a_\la=\overline{G^a (\bC v_\la)}\subset \bP(V_\la^a).
\] 
We note that one has  $\Fl^a_\la=\overline{(N^-)^a\cdot (\bC v_\la)}$, so that the varieties $\Fl^a_\la$
are so-called $\bG_a^{\dim \fn}$-varieties (i.e. the group acts on $\Fl^a_\la$ with the open orbit,
see \cite{HT}, \cite{A}). We also note that in contrast with the classical situation,
the orbit is properly contained in its closure. 

The existence of the embeddings \eqref{Cartan} implies two important properties of the varieties 
$\Fl^a_\la$ in types ${\tt A}$ and ${\tt C}$. 
First, for two dominant weights $\la$ and $\mu$ one has the embedding
of varieties $\Fl^a_{\la+\mu}\hk \Fl^a_\la\times \Fl^a_\mu$ sending the highest weight line $\bC v_{\la+\mu}$
to the product $\bC v_\la\times \bC v_\mu$. Secondly, $\Fl^a_\la\simeq \Fl^a_\mu$ provided $(\la,\omega_i)=0$
if and only if $(\mu,\omega_i)=0$, i.e. the supports of $\la$ and $\mu$ coincide.

By definition, the group $G^a$ acts on the variety $\Fl^a_\la$. 
The group $G^a$ has a natural
one-dimensional extension still acting on $\Fl^a_\la$. Consider the Lie algebra
$\g^a\oplus \bC d$, where $[d,\fb]=0$ and $[d,f_\al]=f_\al$ for any $\al>0$. Then the extended algebra
acts on $V_\la^a$ and $d$ acts as ``a counting the PBW-degree operator", i.e. on $V_\la(s)$ the operator
$d$ acts as a scalar $s$. Let $G^a\rtimes \bC^*$ be the corresponding extended Lie group. Note that
the dimension of the torus of the extended group is increased by one. 

Assume now $G=SL_n$. Then the varieties $\Fl^a_\la$ are known to be flat degenerations of the 
corresponding classical flag varieties. They have analogues of the Pl\" ucker embeddings 
defined by the degenerate Pl\" ucker relations. 
The varieties $\Fl^a_\la$ share some important properties with their classical analogues. In particular, 
$\Fl^a_\la\simeq \Fl^a_\mu$ if and only if $(\la,\omega_i)=0$ implies $(\mu,\omega_i)=0$ and vice versa. 
They also enjoy the following explicit description.
Let $W$ be an $n$-dimensional vector space with a basis $w_1,\dots,w_n$. We define the projection operators
$pr_k:W\to W$, $k=1,\dots,n$  by the formula $pr_k(\sum_{i=1}^n c_iw_i)=\sum_{i\ne k} c_i w_i$.
Fix a collection of numbers $1\le d_1<\dots <d_k\le n$. 
Then $\Fl^a_{\omega_{d_1}+\dots +\omega_{d_k}}$ is isomorphic to the variety of collections $(V_i)_{i=1}^k$ of 
subspaces $V_i\subset W$ satisfying for $i=1,\dots,k-1$
\[
\dim V_i=d_i,\ pr_{d_i+1}\dots pr_{d_{i+1}} V_i\subset V_{i+1}.
\]
In what follows, we sometimes denote this variety by $\Fl^a_{(d_1,\dots,d_k)}$ or simply by
$\Fl^a_\bd$ with $\bd=(d_1,\dots,d_k)$.
For instance, $\Fl^a_{(d)}\simeq \Fl_{\omega_d}\simeq Gr_d(n)$.  

The varieties $\Fl^a_\bd$ are singular projective algebraic varieties with rational singularities. 
They come equipped with line bundles $\cL_\mu$, where $(\mu,\omega_d)=0$ unless $d\in\bd$ 
and $\cL_\mu$ is the pullback of the line bundle $\cO(1)$ under the map $\Fl^a_\bd\to \bP(V^a_\mu)$. 
One of the main tools for the study of algebro-geometric properties of the degenerate flag varieties is
an explicit construction of the resolution of singularities of $\Fl^a_\la$. Namely, 
given a collection $\bd=(d_1,\dots,d_n)$ we denote by $P_\bd$ the set of positive roots in the 
radical of the parabolic subalgebra of $\g$ containing (exactly) the simple roots
$\al_{d_1},\dots,\al_{d_k}$. We define $R_\bd$ as the variety 
of collections of subspaces $V_{i,j}$, $1\le i\le j\le n-1$, $\al_{i,j}\in P_\bd$, satisfying the 
following conditions:
\begin{gather}
\dim V_{i,j}=i,\ V_{i,j}\subset \spa (w_1,\dots,w_i,w_{j+1},\dots,w_n),\\
V_{i,j}\subset V_{i+1,j},\ V_{i,j}\subset V_{i,j+1}\oplus \bC w_{j+1}.
\end{gather}     
Using this desingularization, the following theorem was proved in \cite{FF}.
\begin{thm}
\begin{enumerate}
\item The resolution $R_\bd\to\Fl^a_\bd$ is crepant.
\item The varieties $\Fl^a_\bd$ are normal locally complete intersections (thus Cohen-Macaulay and Gorenstein).
\item The varieties $\Fl^a_\bd$ have rational singularities and are Frobenius split.
\item The cohomology groups $H^m(\Fl^a_\bd,\cL_\mu)$ vanish unless $m=0$ and the zero cohomology is
isomorphic to $(V_\mu^a)^*$. 
\end{enumerate}
\end{thm}

\section{Filtrations: $\msp_{2n}$ vs $\msl_{2n}$}\label{Filtrations}
Let  $w_1,\ldots, w_{2n}$ be a basis of a $2n$-dimensional vector space $W$.
We fix a non-degenerate sympletic form $\langle\cdot,\cdot\rangle$ defined by the conditions
$\langle w_i,w_{2n+1-i}\rangle=1$ for $1\le i\le n$ and $\langle w_i,w_j\rangle=0$ for all
$1\le i,j\le 2n$, $j\not=2n+1-i$. We realize the symplectic group $Sp_{2n}$ as the group of 
automorphisms of $W$ leaving the form invariant. The diagonal matrices
$$
T=\left\{ t= \left(
\begin{array}{ccccc}
t_1 & 0 & 0 & 0 & 0 \\0 & t_2 & 0 & 0 & 0
\\0 & 0 & \ddots & 0 & 0 \\0 & 0 & 0 & t_2^{-1} & 0
\\0 & 0 & 0 & 0 & t_1^{-1}
\end{array}\right)\mid t_1,\ldots,t_n\in\bc^*\right\}
$$
form a maximal torus $T\subset Sp_{2n}$, and the subgroup
$B\subset Sp_{2n}$ of upper triangular matrices  is a Borel subgroup for $Sp_{2n}$.
In such a realization the root vectors of $\msp_{2n}=\text{Lie}\, Sp_{2n}$ are explicitly given 
by the formulas:
\begin{equation}\label{matrix}
f_{i,j}=f_{\al_{i,j}}=
\begin{cases}
E_{j+1,i}-E_{2n+1-i,2n-j},\ 1\le i \le j <n,\\
E_{j+1,i} + E_{2n+1-i,2n-j},\ j\ge  n, i+j<2n,\\
E_{2n+1-i,i},\ 1\le i\le n.
\end{cases}
\end{equation}
As usual, $E_{j,i}$ is the matrix having zero entries everywhere except for the entry 1 in the
$j$-th row, $i$-th column.

Given a dominant weight $\lambda=a_1\omega_1+\ldots+a_n\omega_n$ for the group $Sp_{2n}$, we can consider
this also as dominant weight for the larger group $SL_{2n}\supset Sp_{2n}$. In fact,
let $V_{\lambda,\msl_{2n}}$ be the corresponding irreducible $\msl_{2n}$-representation and
fix a highest weight vector $v_{\msl_{2n}}\in V_{\lambda,\msl_{2n}}$. We identify $V_{\lambda,\msp_{2n}}$
with the irreducible $\msp_{2n}$-submodule of $V_{\lambda,\msl_{2n}}$ generated by $v_{\msl_{2n}}$.

Consider now the action of $\msl_{2n}$ on $V_{\lambda,\msl_{2n}}$. 
Let $\fn^+_{\msl_{2n}},\fn^-_{\msl_{2n}} \subset {\mathfrak{sl}}_{2n}$
respectively be the Lie algebra of the unipotent radical of the Borel subgroup of upper triangular matrices
$B_{SL_{2n}}\subset SL_{2n}$  and of the opposite Borel subgroup $B^-_{SL_{2n}}$. 
Recall the group $SL_{2n}^a=B_{SL_{2n}}\ltimes N_{SL_{2n}}^-$ and the degenerate flag varieties
$\Fl^a_\lambda:=\overline{SL_{2n}^a\cdot v_{\msl_{2n}}}\subset \pc(V_{\lambda,\msl_{2n}}^a)$.

The inclusions $Sp_{2n}\subset SL_{2n}$, $B\subset B_{SL_{2n}}$ and $\nn\subset \fn^-_{\msl_{2n}}$ give 
rise to an action
of $S(\fn^-)$ on $V_{\lambda,\msl_{2n}}^{a}$. 
We want to compare this action with the action on 
$V_{\lambda,\msp_{2n}}^{a}$. Let ${\mathfrak{C}}\subset  V_{\lambda,\msl_{2n}}^{a}$ be the cyclic module 
${\mathfrak{C}_\lambda}:=S(\fn^-).v_{\msl_{2n}}\subset V_{\lambda,\msl_{2n}}$.
\begin{prop}\label{CgleichV}
${\mathfrak{C}_\lambda}\simeq V_{\lambda,\msp_{2n}}^{a}$ as $S(\fn^-)$-module.
\end{prop}
\begin{proof}
Let $\Phi_{\msl_{2n}}^+, \Phi_{\msp_{2n}}^+$ be the sets of positive roots for $\msl_{2n}$ respectively $\msp_{2n}$,
see \eqref{prsln},  \eqref{ij}, \eqref{ibarj}.
We define an injective map $\varphi: \Phi_{\msp_{2n}}^+ \rightarrow \Phi_{\msl_{2n}}^+$ between the sets of
positive roots by the formula $\al_{i,j}\mapsto\al_{i,j}$.
For $\alpha=\al_{i,j}\in \Phi_{\msl_{2n}}^+$ we write $F_{\alpha}$ for the root vector associated
to $-\alpha$.
Note that $F_\alpha=E_{j,i}$.  
For $\alpha\in \Phi_{\msp_{2n}}^+$ we keep the usual notation $f_{\alpha}$ for the root 
vector in $\fn^-$ associated to $-\alpha$. The matrix corresponding to $f_{\alpha_{i,j}}$ is given
in \eqref{matrix}. 
We use the shorthand notations $F_{i,j}=F_{\al_{i,j}}$, $f_{i,j}=f_{\al_{i,j}}$.
Note that since $E_{j,i}$ with $i>n$ acts trivially on $V^a_{\lambda,\msl_{2n}}$ 
($(\la,\omega_k)=0$ for $k>n$), 
the image of $f_\alpha$ in $\text{End}\,V_{\lambda,\msl_{2n}}^{a}$
is equal to the image of $F_{\varphi(\alpha)}$.


The cyclic $\fn^-_{\msl_{2n}}$-module $V_{\lambda,\msl_{2n}}^{a}$ can be described in terms of
generators and relations, see Theorem \ref{genrel}.
Let us write ${\mathfrak{C}_\lambda}$ as a quotient $S(\fn^-)/I_{\mathfrak{C}_\lambda}$.
Then the ideal $I_{\mathfrak{C}_\lambda}$ contains the ideal:
$$
I_\lambda=\left\langle
\begin{array}{l}
U(\fn^+)\circ f_{i,j}^{{a_i+\ldots+a_j+1}}\mid 1\le i \le j\le n-1, \\
U(\fn^+)\circ f_{i,2n-i}^{{a_i+\ldots+a_n+1}}\mid 1\le i\le n.\\
\end{array}
\right\rangle
$$
Theorem \ref{genrel} implies that one has a natural $S(\fn^-)$-equivariant surjective map
$V_{\lambda,\msp_{2n}}^{a}\rightarrow {\mathfrak{C}_\lambda}$. To prove that the
map is an isomorphism, it suffices to show that both modules have the same dimension.
For the proof we use the bases of $V_{\lambda,\msp_{2n}}^{a}$ 
and $V_{\lambda,\msl_{2n}}^{a}$, see Theorem~\ref{basis}.
We define an injective map
$$
\varphi: P_{\msp_{2n}}(\lambda)\rightarrow P_{\msl_{2n}}(\lambda),\quad
(c_\beta)_{\beta\in \Phi_{\msp_{2n}}^+}\mapsto (d_{\gamma})_{\gamma\in \Phi_{\msl_{2n}}^+},
$$
where $(d_{\gamma})_{\gamma\in \Phi_{\msl_{2n}}^+}:=\varphi((c_\beta)_{\beta\in \Phi_{\msp_{2n}}^+}) $ is
defined by the rule:
$$
d_{\gamma}:=
\begin{cases}
c_\beta&\text{if\ } \gamma =\varphi(\beta)\ \text{for some $\beta\in \Phi_{\msp_{2n}}^+$},\\
0 &\text{otherwise}.
\end{cases}
$$
It is now easy to see that
\begin{equation}\label{slbasis}
(d_{\gamma})_{\gamma\in \Phi_{\msl_{2n}}^+}=\varphi((c_\beta)_{\beta\in \Phi_{\msp_{2n}}^+})\in P_{\msl_{2n}}(\lambda)
\end{equation}
is an element in the associated
polytope for $\g=\msl_{2n}$. We know  that the vectors
$$
\{\prod_{\beta\in \Phi_{\msp_{2n}}^+     } f_\beta^{c_\beta}v_{\msp_{2n}}\mid
(c_\beta)_{\beta\in \Phi_{\msp_{2n}}^+}\in P_{\msp_{2n}}(\lambda)\}\subset  V^{a}_{\lambda,\msp_{2n}}
$$
form a basis. We want to show that the vectors
$$
\{\prod_{\beta\in \Phi_{\msp_{2n}}^+     } f_\beta^{c_\beta}v_{\msl_{2n}} \mid
(c_\beta)_{\beta\in \Phi_{\msp_{2n}}^+}\in P_{\msp_{2n}}(\lambda)\}\subset
{\mathfrak{C}_\lambda}\subset V^{a}_{\lambda,\msl_{2n}}
$$
are linearly independent. Recall that $f_\beta$ acts on $V^{a}_{\lambda,\msl_{2n}}$ in the same way as $F_{\varphi(\beta)}$
for the roots $\beta=\al_{i,j}$, $1\le i\le j< n$ and $\beta=\al_{i,2n-i}$, $1\le i\le n$, and for
$\beta=\al_{i,j}$ with $j\ge n$, $i+j<2n$  $f_\beta$ acts as
$F_{i,j}+F_{2n-j,2n-i}$.
Hence for
$(d_{\gamma})_{\gamma\in \Phi_{\msl_{2n}}^+}=\varphi((c_\beta)_{\beta\in \Phi_{\msp_{2n}}^+}) $ we get
\begin{equation}\label{vector}
\begin{array}{rcl}
b_{(c_\beta)}&:=&\prod_{\beta\in \Phi_{\msp_{2n}}^+     } f_\beta^{c_\beta}v_{\msl_{2n}}\\
&=&
(\prod_{i+j<2n, j>n} (F_{i,j}+F_{2n-j,2n-i})^{d_{\alpha_{i,j}}})
(\prod_{1\le i\le j <n} F_{i,j}^{d_{\alpha_{i,j}}})
(\prod_{1\le i\le n} F_{i,2n-i}^{d_{\alpha_{i,2n-i}}})v_{\msl_{2n}}.
\end{array}
\end{equation}
Let $\alpha_1,\ldots,\alpha_{2n-1}$
be the set of simple roots for the root system $\Phi_{\msl_{2n}}$ (Bourbaki enumeration). 
We associate to a vector
$v_{(\ell_\beta)}=\prod_{\beta\in \Phi_{\msl_{2n}}^+ } F_\beta^{\ell_\beta}v_{\msl_{2n}}$
the collection
$$
\zeta(v_{(\ell_\beta)})=(r_1,\ldots,r_{2n-1})\quad\text{such that\ } \sum_{\beta\in \Phi_{\msl_{2n}}^+ }
\ell_\beta \beta=\sum_{i=1}^{2n-1}r_i\alpha_i,
$$
called the {\it root weight}. We define a partial order on these vectors by:
$$
v_{(\ell_\beta)} >v_{(k_\beta)}\quad\text{ if }\quad
\zeta(v_{(\ell_\beta)})> \zeta(v_{(k_\beta)})
$$
with respect to the induced lexicographic order on the $(2n-1)$-tuples.
It follows for the vector $b_{(c_\beta)}$
in $(\ref{vector})$:
\begin{equation}\label{develop}
b_{(c_\beta)} = v_{(d_\alpha)} + \text{\, a sum $\sum x_{(\ell_\beta)}v_{(\ell_\beta)}$ of smaller terms},
\end{equation}
because if one chooses in a factor in $(\ref{vector})$ the factor $F_{2n-j,2n-i}$ instead of 
$F_{i,j}$, then the corresponding $(2n-1)$-tuple is strictly smaller with respect to the
partial order above. It follows in particular
that $b_{(c_\beta)}\not= 0$ because $v_{(d_\alpha)}$ is a basis vector for 
$V_{\lambda,\msl_{2n}}^{a}$ by $(\ref{slbasis})$
and the smaller summands are weight vectors of different weights.

The linear independence of the vectors $\{b_{(c_\beta)}\mid (c_\beta)\in P_{\msp_{2n}}(\lambda)\}$,
follows along the same lines: given a linear
dependence relation for the $b_{(c_\beta)}$, choose a maximal element $(p_1,\ldots,p_{2n-1})$ among the $\zeta(v_{(\varphi(c_\beta)})$.
By weight reasons the linear independence of the $b_{(c_\beta)}$ implies a linear independence relation
between the
summands of associated root weight $(p_1,\ldots,p_{2n-1})$.
By maximality, a summand of $b_{(c_\beta)}$ as in $(\ref{develop})$ has root weight $(p_1,\ldots,p_{2n-1})$
if and only if $v_{(d_\alpha)}$ has root weight $(p_1,\ldots,p_{2n-1})$, so we obtain a linear dependence
relation between the $v_{(d_\alpha)}$, which is not possible.

It follows: $\dim {\mathfrak{C}_\lambda}= \dim V_{\lambda,\msp_{2n}}^{a}$ and the canonical
map $V_{\lambda,\msp_{2n}}^{ab}\rightarrow {\mathfrak{C}_\lambda}$ is an isomorphism.
\end{proof}

\begin{cor}
$\SF_\lambda\subset \Fl_\lambda$ for any $\la$ of the form $\sum_{i=1}^n a_i\omega_i$.
\end{cor}
\begin{proof}
By Proposition~\ref{CgleichV}, we can identify
$\SF^a_\lambda:=\overline{Sp_{2n}^{a}\cdot v_{Sp_{2n}}}\subset \pc(V_{\lambda,Sp_{2n}}^{a})$
with $\overline{Sp_{2n}^a\cdot v_{\msl_{2n}}}\subset \pc({\mathfrak{C}_\lambda})$ and hence
$$
\SF^a_\lambda=\overline{Sp_{2n}^a\cdot v_{\msl_{2n}}}\subset 
\overline{SL_{2n}^a\cdot v_{\msl_{2n}}}
=\Fl^a_\lambda\subset \pc(V_{\lambda,\msl_{2n}}^a).
$$
\end{proof}

\section{Explicit description}\label{Explicit}
In this section we give an explicit description of the symplectic degenerate flag  varieties
in the linear algebra terms. We start with the case of symplectic Grassmannians.  

\subsection{The degenerate symplectic Grassmann variety}
Let $V_{\omega_k}$ be the irreducible fundamental $Sp_{2n}$-representation
of highest weight $\omega_k$ (Bourbaki enumeration). In particular, $V_{\omega_1}\simeq W$ is the
standard vector representation.
Let $V_{\omega_k}\hookrightarrow \Lambda^k W$
be the canonical embedding defined by mapping a fixed highest weight vector to $w_1\wedge\ldots\wedge w_k$.
In the following we will identify $V_{\omega_k}$ with the image.
By definition, the symplectic degenerate Grassmann variety $SpGr^a_k(W)\subset \bP(V_{\omega_k}^a)$ 
is equal to $\SF^a_{\omega_k}$.  

Let $P\subset Sp_{2n}$ be the maximal parabolic subgroup
associated to $\omega_k$. Set $\p=\text{Lie}\,P$ and let $\g=\n^-\oplus \lfr\oplus\n$ be a
decomposition such that
$\lfr$ is the Lie algebra of the Levi subgroup of $P$ containing $T$, $\p=\lfr\oplus\n$,
and $\n^-$ is the Lie algebra
of the unipotent radical of the parabolic subgroup $P^-$ opposite to $P$. The Lie algebra $\n^-$ consists
of matrices of the form
\begin{equation}\label{liealg}
\left(\begin{array}{cccc}
0 & 0 & 0 & 0  \\
A & 0 & 0 & 0  \\
B & 0 & 0 & 0  \\
C & B^{tn} & -A^{nt} & 0
\end{array}\right),
\end{equation}
where $A,B$ are $(n-k)\times k$ matrices, $X^{nt}$ denotes the transposed matrix with respect to the
skew diagonal,
and $C$ is a $k\times k$ matrix such that $C^{nt}=C$.

We write $W=W_{k,1}\oplus W_{k,2}\oplus W_{k,3}$, where
$$
W_{k,1}=\text{span}(w_1,\ldots, w_k),\ W_{k,2}=\text{span}(w_{k+1},\ldots, w_{2n-k}),
\ W_{k,3}=\text{span}(w_{2n-k+1},\ldots, w_{2n}).
$$

Denote by $p_{1,3}$ the projection $p_{1,3}:W\rightarrow W_{k,1}\oplus W_{k,3}$, i.e.,
$$
pr_{1,3}(x_1,\ldots,x_{2n})
=(x_1,\ldots,x_k,0,\ldots,0,x_{2n-k+1},\ldots,x_{2n}).
$$
\begin{prop}\label{ZKsympgrass}
$$
SpGr^a_k(W)=\{U\in Gr_k(W)\mid pr_{1,3}(U)\ \text{is isotropic} \}.
$$
\end{prop}
Denote by $Z_k$ the subvariety $\{U\in Gr_k(W)\mid pr_{1,3}(U)\ \text{is isotropic} \}$.
The first simple observation in the proof of the proposition is:
\begin{lem}\label{grassgleichZk}
$Gr_k(W_{k,1}\oplus W_{k,2})\subset Z_k$.
\end{lem}
\begin{proof}
Since $pr_{1,3}(U)\subset W_{k,1}$ for all $U\in Gr_k(W_{k,1}\oplus W_{k,2})$, we have 
$pr_{1,3}(U)$ isotropic and hence $U\in Z_k$.
\end{proof}
Denote by $Sp_{2k}\subset Sp_{2n}$ the symplectic subgroup acting only on the first and last $k$ coordinates.
The matrix for an element of this subgroup looks like
$$
\left(\begin{array}{cccc}A & 0 & 0 & B \\
0 & 1\hskip -2.5pt \hbox{I} & 0 & 0 \\
0 & 0 & 1\hskip -2.5pt \hbox{I} & 0 \\C & 0 & 0 & D\end{array}\right),
$$
where $A,B,C,D$ are $k\times k$ matrices,
$$
\left(\begin{array}{cc}A & B \\C & D\end{array}\right)
$$
is a $2k\times 2k$-sympletic matrix, and $1\hskip -2.5pt \hbox{I}$ is an 
$(n-k)\times (n-k)$ identity matrix.
Further, let
$$
P_k:=\{ g\in Sp_{2k}\mid C=0\},
$$
then $P_k$ is the maximal parabolic subgroup of $Sp_{2k}$ associated to the long simple root.
The next simple observation is:
\begin{lem}
$Z_k=Sp_{2k}\cdot Gr_k(W_{k,1}\oplus W_{k,2})$.
\end{lem}
\begin{proof}
Note first that $Y=Gr_k(W_{k,1}\oplus W_{k,2})$ is stable under the action of $P_k$, so 
$Sp_{2k}\cdot Y$ is closed.
Since $Z_k$ is $Sp_{2k}$-stable, by Lemma~\ref{grassgleichZk} we have $Sp_{2k}\cdot Y\subset Z_k$.
Now assume $U\in Z_k$ and let
$\bar U=pr_{1,3}(U)$. This subspace of $W_{k,1}\oplus W_{k,3}$ is isotropic by assumption, 
so there exists a $g\in Sp_{2k}$
such that $g\cdot \bar U\subset W_{k,1}$, and hence $g\cdot U\subset W_{k,1}\oplus W_{k,2}$. 
It follows $g\cdot U\in Gr_k(W_{k,1}\oplus W_{k,2})$ and hence 
$Z_k=Sp_{2k}\cdot Gr_k(W_{k,1}\oplus W_{k,2})$.
\end{proof}
As an immediate consequence one sees that $Z_k$ is the image of the canonical product map
$$
\tilde \pi: Sp_{2k}\times  Gr_k(W_{k,1}\oplus W_{k,2})\rightarrow Z_k\subset Gr_k(W),
\quad (g,U)\mapsto g\cdot U,
$$
and hence:
\begin{cor}
$Z_k$ is irreducible.
\end{cor}
\noindent
{\it Proof of Proposition~\ref{ZKsympgrass}.}
Recall the description of the Lie algebra $\n^-$ in $(\ref{liealg})$
and the inclusion of the abelianized action described in Proposition~\ref{CgleichV}.
It follows that if $\gamma\in \n^-$ is a matrix as in $(\ref{liealg})$,
then in $V_{\omega_k}^a$ we have $\exp \gamma\cdot [w_1\wedge \ldots \wedge w_k]$ is the $k$-dimensional
subspace having as basis the column vectors of the matrix
\begin{equation}\label{liealg3}
\left(\begin{array}{c}
1\hskip -3pt \text{I}  \\
A   \\
B   \\
C
\end{array}\right).
\end{equation}
The symmetry condition of the matrix $C$ implies that these subspaces lie in $Z_k$.
To prove equality, for  $\ui=(i_1,\ldots,i_k)$, $1\le i_1< \ldots <i_k\le 2n$,
let $w_\ui=w_{i_1}\wedge\ldots\wedge w_{i_k}\in \Lambda^k W$
and denote by $p_\ui$ the corresponding Pl\"ucker coordinate,
i.e. $p_\ui(w_\uj)=\delta_{\ui,\uj}$.

Consider in $Z_k$ the open affine set
$$
{Z_k}_{(1,2,\ldots,k)}=\{U \in Z_k\mid p_{(1,2,\ldots,k)}(U)\not=0 \}.
$$
The spaces having a basis as in $(\ref{liealg3})$ lie in ${Z_k}_{(1,2,\ldots,k)}$, so this set is
non-empty and (since $Z_k$
is irreducible) dense. Now given an element in ${Z_k}_{(1,2,\ldots,k)}$, one can find a basis
corresponding to the columns of a matrix of the form
\begin{equation}\label{liealg4}
\left(\begin{array}{c}
1\hskip -3pt \text{I}  \\
A   \\
B   \\
C
\end{array}\right),
\end{equation}
where $A,B$ are $(n-k)\times k$ matrices, and the condition $pr_{1,3}(U)$ is isotropic implies that
$C=C^{nt}$. It follows that ${Z_k}_{(1,2,\ldots,k)}=Sp_{2n}^a\cdot [w_1\wedge\ldots\wedge w_k]$,
and hence $Z_k= \ SpGr^a_k(W)$.
\qed

\begin{rem} {\it a)} It is easy to check that the canonical map 
$$\pi: Sp_{2k}\times_{P_k} Gr_k(W_{k,1}\oplus W_{k,2})\rightarrow \ SpGr^a_k(W)$$
is a desingularization.
\par
{\it b)} Denote by $Z_{j}$, $0\le j\le k$, the subset
$$
Z_{j}:=\{U\in SpGr^a_k(W)\mid \dim pr_{1,3}(U)=j\}
$$
and define 
$$
Gr_k^j(W_{k,1}\oplus W_{k,2}):=\{U\in Gr_k(W_{k,1}\oplus W_{k,2})\mid \dim pr_{1,3}(U)=j\}.
$$
We have obviously a partition:
$$
SpGr_k(W):=\bigcup_{j=0,\ldots,k}Z_{j}.
$$ 
One can show: ${Z_{j}}$ is smooth,
$\overline{Z_{j}}=\bigcup_{i=0,\ldots,j}Z_{j}$, the desingularization map above
is compatible with the partition and induces maps
$$
\pi^j: Sp_{2k}\times_{P_k} Gr^j_k(W_{k,1}\oplus W_{k,2})\rightarrow \ Z_{j}
$$
such that the fibres of $\pi^j$ are all of dimension $\frac{1}{2}(k-j)(k-j+1)$.
In particular, the map $\pi$ is semismall but not small (failure for $j=k-1$).
\end{rem}

\subsection{Symplectic degenerate flag varieties}
Based on Proposition \ref{ZKsympgrass} we derive explicit description for all symplectic degnerate flag
varieties. 
For a regular dominant weight $\la$ we denote the complete symplectic degenerate flag variety 
$\Fl^a_\la$ by $\SF^a_{2n}$.  

Recall the basis $w_1,\dots,w_{2n}$ of $W$. We denote by $pr_i:w\to W$ the projections along the $w_i$, i.e.
$pr_i (\sum_{j=1}^{2n} c_jw_j)=\sum_{j\ne i} c_jw_j$.
\begin{thm}
The degenerate symplectic flag variety $\SF^a_{2n}$ is naturally embedded into the product
$\prod_{i=1}^n SpGr^a_i(2n)$ of degenerate symplectic Grassmannians. The image of the
embedding is equal to the set of collections $(V_i)_{i=1}^n$, $V_i\in SpGr^a_i(2n)$ satisfying
the conditions
\[
pr_{i+1} V_i\subset V_{i+1}, i=1,\dots,n-1.
\]
\end{thm}
\begin{proof}
According to the proof of Proposition \ref{ZKsympgrass} 
we have 
\[
\SF^a_{2n}\cap \prod_{i=1}^n Z_{i(1,\dots,i)}=\{(V_1,\dots,V_i):\ V_i\in Z_{i(1,\dots,i)}, 
pr_{i+1}V_i\subset V_{i+1}\}.
\]
Moreover this subvariety of $\SF^a_{2n}$ coincides with the $Sp^a_{2n}$-orbit of the highest weight line.
Thus we only need to show that the variety defined above is irreducible. 
This is proved in Corollary \ref{irred}, using the desingularization for $\SF^a_{2n}$.
\end{proof}

For a subspace $V\subset W$ we denote by $V^\perp\subset W$ the orthogonal 
complement to $V$. 
Define an order two automorphism $\sigma\in Aut(\prod_{i=1}^{2n-1} Gr_i(2n))$ by the formula
\[
\sigma (V_i)_{i=1}^{2n-1}=(V_{2n-1}^\perp,V_{2n-2}^\perp,\dots,V_1^\perp).
\]
\begin{prop}\label{sigma}
The automorphism $\sigma$ defines an order two automorphism of the complete degenerate flag variety
$\Fl^a_{2n}$. The set of $\sigma$-fixed points $(\Fl^a_{2n})^\sigma$ is isomorphic to the complete
symplectic degenerate flag variety $\SF^a_{2n}$.
\end{prop}
\begin{proof}
Follows from the definition.
\end{proof}

Let $\bd=(d_1,\dots,d_k)$ be a collection of integers such that
$1\le d_1<\dots <d_k\le n$. Let $\SF^a_\bd$ be the parabolic degenerate 
flag variety, corresponding to the highest weight $\sum_{i=1}^k \omega_{d_i}$.
\begin{thm}
The parabolic degenerate symplectic flag variety $\SF^a_\bd$ is naturally 
embedded into the product
$\prod_{i=1}^k SpGr^a_{d_i}(2n)$ of degenerate symplectic Grassmannians. The image of the
embedding is equal to the variety of collections $(V_i)_{i=1}^k$, $V_i\in SpGr^a_{d_i}(2n)$ satisfying
the conditions
\[
pr_{d_i+1}\dots pr_{d_{i+1}} V_{d_i}\subset V_{d_{i+1}}, i=1,\dots,k-1.
\]
\end{thm}
\begin{proof}
As in above, we only need to show the irreducibility of $\SF^a_\bd$. This is proved in Corollary
\ref{parirred}.
\end{proof}

Let us fix a collection $\bd=(d_1,\dots,d_k)$, $1\le d_1<\dots <d_k\le n)$.
Assume $d_k<n$. Then we define an extended collection ${\bf D}$ by the formula
\[
{\bf D}=(d_1,\dots,d_k,2n-d_k,\dots,2n-d_1). 
\]
If $d_k=n$, then we set ${\bf D}=(d_1,\dots,d_k,2n-d_{k-1},\dots,2n-d_1)$.  
We define an order two automorphism 
$\sigma_\bd\in Aut(\prod_{d\in {\bf D}} Gr_d(2n)$ 
by the formula
\[
\sigma_\bd (V_i)_{i=1}^{l}=(V_l^\perp,V_{l-1}^\perp,\dots,V_1^\perp),
\]
where $l=2k$ if $d_k<n$ and $l=2k-1$ otherwise.
\begin{prop}
The automorphism $\sigma_\bd$ defines an order two automorphism of the 
$SL_{2n}$ parabolic degenerate flag variety
$\Fl^a_{\bf D}$. The set of $\sigma_\bd$-fixed points $(\Fl^a_{\bf D})^{\sigma_\bd}$ is isomorphic to 
the parabolic symplectic degenerate flag variety $\SF^a_{\bf D}$.
\end{prop}

\subsection{The degeneration}
In this subsection we prove that the varieties $\SF^a_\bd$ are flat degenerations of their classical
analogues. 
\makeatletter
\def\Ddots{\mathinner{\mkern1mu\raise\p@
\vbox{\kern7\p@\hbox{.}}\mkern2mu
\raise4\p@\hbox{.}\mkern2mu\raise7\p@\hbox{.}\mkern1mu}}
\makeatother
Let $J_s$ be the $2n\times 2n$-matrix
$$
\left(\begin{array}{cccc}
0 & 0 & 0 & I_k \\
0 & 0 & sI_{n-k} & 0 \\
0 & -sI_{n-k}& 0 & 0 \\
-I_k & 0 & 0 & 0
\end{array}\right),\quad\text{where $I_l$ is a $l\times l$ matrix}\
\left(\begin{array}{ccc}
0 & 0  & 1 \\
0 &  \Ddots & 0 \\
1 & 0 & 0\end{array}\right).
$$
The matrix defines a non-degenerate symplectic form for $s\not=0$, the group
$Sp_{2n}$ is the one leaving invariant the form for $s=1$, and
$$
SpGr_k(2n)=\{U\in Gr_k(\bc^{2n}) \mid U \text{\ is isotropic with respect to $J_0$} \}.
$$
Denote by $\eta$ the following one-parameter subgroup:
$$
\eta:\bc^*\rightarrow D=\text{\ diagonal matrices in $GL_{2n}$},\quad s\mapsto
\left(\begin{array}{ccc}
1\hskip -3pt \text{I}_{k} & 0 & 0 \\
0 & s1\hskip -3pt \text{I}_{2(n-k)} & 0 \\
0 & 0 & 1\hskip -3pt \text{I}_{k}\end{array}\right).
$$
then $\eta(s)^t J_1\eta(s)=J_{s^2}$, and it follows that if $U\in SpGr_k(\bc^{2n})$ is a subspace
isotropic with respect to $J_1$, then $\eta(s^{-1})(U)$ is isotropic with respect to $J_{s^2}$.

Recall that $Sp_{2n}/P_k$ is sitting in the Grassmann variety as the set
of isotropic subspaces. Consider
\begin{equation}\label{Y}
Y:=\overline{\{(\eta(s^{-1})(U),s)\mid s\in\bc^{*},U\in Sp_{2n}/P_k\}}\subset Gr_k(\bc^{2n})\times \bc
\end{equation}
together with the natural projection $\phi:Y\rightarrow \bc$ onto 
the second factor.
\begin{prop}
The projection map $\phi$ is flat, $\phi^{-1}(s)\simeq Sp_{2n}/P_k$ for all $s\not=0$ and $\phi^{-1}(0)=SpGr^a_k(\bc^{2n})$.
\end{prop}
\begin{proof}
The map is flat because $Y$ is irreducible and
$\phi$ is dominant (\cite{H}, Chap.~III, Proposition 9.7).
It remains to show that the fibres are the spaces described above. We assume without
loss of generality: $k\ge 2$.
Consider the map
$$
\begin{array}{ccl}
\Psi:\Lambda^k\bc^{2n}\times \bc &\rightarrow &\Lambda^{k-2}\bc^{2n},\\
(v_1\wedge\ldots\wedge v_k,s)&\mapsto&\sum_{\ell<m} Q_s(v_\ell,v_m)(-1)^{m+\ell-1}
v_1\wedge\ldots \wedge \hat v_\ell\wedge\ldots\wedge \hat v_m \wedge\ldots \wedge v_k
\end{array}
$$
where $Q_s(v,w)=v^t J_s w$. The map is homogeneous with respect to $\Lambda^k\bc^{2n}$, the preimage
$\Psi^{-1}(0)$ defines hence a closed subset of $\pc(\Lambda^k\bc^{2n})\times \bc$, and the 
intersection with the Grassmann variety defines a closed subset $\tilde Y\subset Gr_k(\bc^{2n})\times \bc$.
Let $\sigma:\tilde Y\rightarrow \bc$ be the projection onto the second 
component. By the definition of $\Psi$
one has $(U,s)\in \sigma^{-1}(s)\subset \tilde Y$ if and only if 
$\Psi(v_1\wedge\ldots\wedge v_k,s)=0$ for a basis $\{v_1,\ldots,v_k\}$ of $U$, 
that is iff $U$ is an isotropic subspace with respect
to the form $Q_s$. It follows that $\tilde Y$ and $Y$ have in common the subset
$$
Y'=\{(\eta(s^{-1})(U),s)\mid s\in\bc^{*},U\in Sp_{2n}/P_k\}\subset Y.
$$
so $Y=\overline{Y'}$ is an irreducible component of $\tilde Y$. Since $\sigma^{-1}(0)=SpGr^a_k(\bc^{2n})$,
is follows
that $\sigma$ has equidimensional and irreducible fibers, and hence $\tilde Y$ is irreducible, which implies
that $Y=\tilde Y$ and $\sigma=\phi$.
\end{proof}

Let $\bd=(d_1,\dots,d_m)$.
\begin{thm}
The varieties $\SF^a_\bd$ are flat degenerations of their classical analogues $\SF_\bd$.
\end{thm}
\begin{proof}
Let us denote by $Y_k$ the variety defined by \eqref{Y}. 
Recall the projections $\phi_k:Y_k\to \bC$. Let $Y_\bd$ be the fibered product of all $Y_{d_i}$ 
over $\bC$, i.e.
\[
Y_\bd=\{(y_1,\dots,y_m):\ y_i\in Y_{d_i}, \phi_{d_1}(y_1)=\dots = \phi_{d_m}(y_m)\}.
\]
Recall 
(see \cite{Fe1}) that there exists
flat degeneration $\psi_\bd:M_\bd\to \bc$ of the $SL_{2n}$ flag variety $\Fl_\bd$ into the degenerate version $\Fl^a_\bd$, so $\psi^{-1}_\bd (0)\simeq \Fl^a_\bd$ and the general fiber is isomorphic to  $\Fl_\bd$.
We note that both $M_\bd$ and $Y_\bd$ are subvarieties in the product 
$\bC\times\prod_{i=1}^m Gr_{d_i}(\bc^{2n})$. Consider the intersection $M_\bd\cap Y_\bd$ and the natural
projection $\varphi_\bd:M_\bd\cap Y_\bd\to \bC$.
Then the general fiber is isomorphic to the classical symplectic flag variety $\SF_\bd$ and 
$\varphi^{-1}_\bd(0)\simeq \SF^a_\bd$. Since $M_\bd\cap Y_\bd=\overline{\varphi_\bd^{-1}(\bc^*)}$, 
the left hand side is irreducible. Now \cite{H}, Chap.~III, Proposition 9.7 implies the theorem.
\end{proof}

\section{Resolution of singularities}\label{Resolution}
In this section we construct varieties $SpR_\bd$ which serve as 
desingularizations for 
symplectic degenerate flag varieties $\SF^a_\bd$. We start with the case of complete flags.

\subsection{Complete flag varieties.}
Recall the $2n$-dimensional vector space $W$ with a basis $w_i$, $i=1,\dots,2n$.
Let $W_{i,j}$ be the linear span of the vectors 
$w_1,\dots,w_i$, $w_{j+1},\dots,w_{2n}$.

We define $SpR_{2n}$ as the variety of collections of subspaces
\[
V_{i,j},\ 1\le i\le j\le 2n,\ i+j\le 2n
\]
subject to the conditions
\begin{gather}
\label{1}
\dim V_{i,j}=i,\ V_{i,j}\subset W_{i,j},\\
\label{2}
V_{i,j}\subset V_{i+1,j},\ pr_{j+1}V_{i,j}\subset V_{i,j+1},\\
\label{3}
V_{i,2n-i} \text{ are isotropic for } i=1,\dots,n.
\end{gather}
\begin{rem}
It is often convenient to view the spaces $V_{i,j}$ as being attached to the positive roots
$\al_{i,j}$ of $\msp_{2n}$ (see \eqref{ij}, \eqref{ibarj}). So in what follows we sometimes
denote $V_{i,j}$ by $V_{\al_{i,j}}$. 
\end{rem}
We note that there is a natural embedding $SpR_{2n}\subset \prod Gr_i(W_{i,j})$, where the
product is taken over $i\le j$, $i+j\le 2n$.

Let us order the set of positive roots of $\msp_{2n}$ into a sequence $\beta_1,\beta_2,\dots,\beta_{n^2}$ as follows:
\begin{gather*}
\beta_1=\al_{1,2n-1},\\ \beta_2=\al_{1,2n-2},\\
\beta_3=\al_{2,2n-2},\ \beta_4=\al_{1,2n-3},\\
\beta_5=\al_{2,2n-3},\ \beta_6=\al_{1,2n-4},\\
\beta_7=\al_{3,2n-3},\ \beta_8=\al_{2,2n-4},\ \beta_9=\al_{1,2n-5},
\end{gather*}
and, finally,
\[
\beta_{n^2-n+1}=\al_{n,n},\ \beta_{n^2-n+2}=\al_{n-1,n-1},\cdots, \beta_{n^2}=\al_{1,1}.
\]
For $l=1,\dots,n^2$ we define the step $l$ variety $SpR_{2n}(l)$ as the image of projection
\[
(V_{\beta_k})_{k=1}^{n^2}\mapsto (V_{\beta_k})_{k=1}^l.
\]
\begin{lem}
The variety $SpR_{2n}(l)$ consists of collections $(V_{\beta_k})_{k=1}^l$ satisfying conditions 
\eqref{1}, \eqref{2}, \eqref{3} whenever the pairs of indices $(i,j)$ appearing in the conditions
satisfy $\al_{i,j}=\beta_m$ with $m\le l$. 
\end{lem}
For example, $SpR_{2n}(n^2)=SpR_{2n}$. We also set $SpR_{2n}(0)=pt$.

\begin{lem}\label{P1}
For all $l=1,\dots,n^2$ the map $SpR_{2n}(l)\to SpR_{2n}(l-1)$ is a $\bP^1$-fibration.
\end{lem}
\begin{proof}
Fix a point $(V_{\beta_k})_{k=1}^{l-1}\in SpR_{2n}(l-1)$. We need to show that 
the possible choices of $V_{\beta_l}$ are parametrized by points of $\bP^1$. 
Let $\beta_l=\al_{i,j}$. First, let $i+j=2n$. If $i=1$, then we are looking for a one-dimensional
subspace in $\spa (w_1,w_{2n})$ (it is automatically isotropic). This gives the isomorphism
$SpR_{2n}(1)\simeq \bP^1$. Now assume $i>1$ and still $i+j=2n$. Then we have the following conditions 
for the space $V_{i,2n-i}$:
\[
V_{i-1,2n-i}\subset V_{i,2n-i}\subset W_{i,2n-i},\ V_{i,2n-i} \text{ is isotropic}.
\] 
We note that $V_{i-1,2n-i}\subset \spa(w_1,\dots,w_{i-1},w_{2n-i+1},\dots,w_{2n})$ is an 
$(i-1)$-dimen\-sional
vector space such that $pr_{2n-i+1} V_{i-1,2n-i}$ is isotropic. Since $V_{i,2n-i}$ has to be isotropic,
the datum of an $i$-dimensional $V_{i,2n-i}\subset W_{i,j}$ is equivalent to 
the datum of a line in the 2-dimensional space
$W_{i,2n-i}/V^\perp_{i-1,2n-i}$. Therefore, the fiber of the map
$SpR_{2n}(l)\to SpR_{2n}(l-1)$ is $\bP^1$.

Now assume $i+j<2n$ and $\beta_l=\al_{i,j}$. Then the conditions for $V_{i,j}$ are 
\[
V_{i-1,j}\subset V_{i,j}\subset W_{i,j}\cap (V_{i,j+1}\oplus \bC w_{j+1}).
\]  
It is easy to see that the variety of choices of such a  $V_{i,j}$ is isomorphic to $\bP^1$ 
(actually, the same situation appears in the  $\msl_{2n}$ case).
\end{proof}

\begin{cor}
The variety $SpR_{2n}$ is a tower of successive $\bP^1$-fibrations.
\end{cor}

We now construct sections of the maps $SpR_{2n}(l)\to SpR_{2n}(l-1)$ similar to the ones
in type ${\tt A}$ (see \cite{FF}). A section $s_l:SpR_{2n}(l-1)\to SpR_{2n}(l)$ 
is determined by the $\beta_l$-th component
of the image $s_l(\bV)$ for $\bV\in SpR_{2n}(l-1)$. So let $\beta_l=\al_{i,j}$. Then we set
\[
(s_l(\bV))_{i,j}=
\begin{cases}
\bC w_{j+1},\text{ if } i=1,\\
V_{i-1,j+1}\oplus \bC w_{j+1}, \text{ if } i>1.
\end{cases}
\]
We sometimes denote $s_l$ by $s_{i,j}$.
Let us denote by $Z_{i,j}\subset SpR_{2n}$ the divisor corresponding to the 
above section, i.e.
\[
Z_{i,j}=\{\bV\in SpR_{2n}:\ V_{i,j}\in (\mathrm{Im} s_{i,j})_{i,j}\}.
\]
We also denote by $Z^\circ_{i,j}\subset Z_{i,j}$ an open part of $Z_{i,j}$ which is the complement to the 
intersection of $Z_{i,j}$ with all other divisors:
\begin{equation}\label{circ}
Z_{i,j}^\circ=Z_{i,j}\setminus \cup_{(i_0,j_0)\ne (i,j)} Z_{i_0,j_0}.
\end{equation}

Let us denote by $SpR_{2n}^\circ\subset SpR_{2n}$ the complement in $SpR_{2n}$ to the union of the 
divisors $Z_{i,.j}$.
\begin{lem}
$SpR_{2n}^\circ$ is an open cell in $SpR_{2n}$ enjoying the following explicit description:
\[
SpR_{2n}^\circ=\{\bV\in SpR_{2n}:\ p_{1,\dots,i}V_{i,j}\ne 0 \text{ for all } i,j\},
\]
where $p_{1,\dots,i}$ is the corresponding Pl\" ucker coordinate.
\end{lem}
\begin{proof}
First we prove that for any $\bV\in SpR_{2n}^\circ$ one has 
\begin{equation}\label{bound}
\dim pr_{i+1}\dots pr_{2n-i} V_{i,i} =i,\ i=1,\dots,n.
\end{equation}
In fact, let $i$ be the minimal number with the property that there exists
$i+1\le j\le 2n-i$ such that $\dim pr_{i+1}\dots pr_j V_{i,i}=i-1$.  For such
an $i$ let us take minimal $j$ with the property as above. Then
since $pr_{i+1}\dots pr_j V_{i,i}\subset V_{i,j}$ we obtain $w_j\in V_{i,j-1}$.
Because of the assumption of the minimality of $i$, $w_j\notin V_{i-1,j-1}$. Hence,
$\dim pr_j V_{i-1,j-1}=i-1$ and $V_{i-1,j}=pr_j V_{i-1,j-1}$. In addition,
$V_{i,j-1}=V_{i-1,j}\oplus \bC w_j$. This implies $V_{i,j-1}=V_{i-1,j}\oplus\bc w_j$, which
means $\bV\in Z_{i,j-1}$.

Second, we claim that \eqref{bound} implies
\begin{equation}\label{p}
\dim pr_{i+1}\dots pr_{2n} V_{i,i} =i,\ i=1,\dots,n.
\end{equation}
In fact, since $\dim pr_{i+1}\dots pr_{2n-i} V_{i,i} =i$, we obtain
\[
V_{i,2n-i}=pr_{i+1}\dots pr_{2n-i} V_{i,i}.
\]
Recall that $V_{i,2n-i}$ is an $i$-dimensional isotropic subspace in 
$W_{i,2n-i}$. This implies the claim.   

Finally, since for any $i,j$ we have $V_{i,j}=pr_{i+1}\dots pr_j V_{i,i}$ and
\eqref{p} holds, we obtain $p_{1,\dots,i}V_{i,j}\ne 0$. 
\end{proof}

We have a natural map $\pi_{2n}:SpR_{2n}\to \SF^a_{2n}$
\[
\pi_{2n} (V_{i,j})_{i,j}=(V_{i,i})_{i=1}^n
\]
forgetting the off-diagonal elements.

\begin{lem}\label{surj}
The map $\pi_{2n}$ is a birational isomorphism.
\end{lem}
\begin{proof}
We prove that $\pi_{2n}$ is surjective and one-to-one on $SpR_{2n}^\circ$. 
The second statement is simple: since $\dim pr_{i+1}\dots pr_j V_i =i$ for any
$j=i+1,\dots,2n-i$ we have $V_{i,j}=pr_{i+1}\dots pr_j V_i$ (recall $\dim V_{i,j}=i$).
Let us prove the surjectivity.

Given a point $(V_i)_{i=1}^n\in\SF^a_{2n}$, we need to construct a point $\bV\in SpR_{2n}$ such that
$\pi_{2n} \bV=(V_i)_{i=1}^n$. We construct the entries $V_{i,j}$ of $\bV$ by increasing induction on 
$j$. 
In the course of the construction, we always check that $pr_{j+1}\dots pr_{2n-i} V_{i,j}$ is isotropic
(this is a necessary condition for  $\bV\in SpR_{2n}$). 
For $j=1$ we have the only space $V_{1,1}$, which has to be equal to $V_1$. 

Now assume $j\le n$. 
We need to define $V_{i,j}$ with $i=1,\dots,j$. We do this by decreasing induction on $i$.
For $i=j$ we have $V_{j,j}=V_j$. Now assume $V_{i+1,j}$ is already defined . Then the conditions
for $V_{i,j}$ are as follows:
\[
pr_j V_{i,j-1}\subset V_{i,j}\subset W_{i,j}\cap V_{i+1,j},\ pr_{j+1}\dots pr_{2n-i} V_{i,j} \text{ is isotropic}.
\]
If $\dim pr_j V_{i,j-1}=i$, then we have $V_{i,j}=pr_j V_{i,j-1}$. Assume that the dimension drops, i.e.
$w_j\in V_{i,j-1}$. Then, by induction $pr_j V_{i,j-1}\subset W_{i,j}\cap V_{i+1,j}$. In addition
$\dim (W_{i,j}\cap V_{i+1,j})\ge i$. Therefore there exists an $i$-dimensional subspace in between
$pr_j V_{i,j-1}$ and $W_{i,j}\cap V_{i+1,j}$. Now since $pr_{j+1}\dots pr_{2n-i-1} V_{i+1,j}$ is isotropic
and $V_{i,j}\subset W_{i,j}\cap V_{i+1,j}$, we obtain that 
$pr_{j+1}\dots pr_{2n-i} V_{i,j}$ is also  isotropic. 

Finally, assume $j>n$. 
We need to define $V_{i,j}$ with $i=1,\dots,2n-j$. Again, we do this by decreasing induction on $i$.
Let us start with $i=2n-j$. Then the conditions we have are
\[
pr_{j+1} V_{2n-j,j-1}\subset V_{2n-j,j}\subset W_{2n-j,j},\ V_{2n-j,j}\text{ is isotropic}. 
\]
If $\dim pr_{j+1}V_{2n-j,j-1}=2n-j$, then we are done, since $pr_{j+1}V_{2n-j,j-1}$ is isotropic by induction.
If $\dim pr_{j+1}V_{2n-j,j-1}=2n-j-1$, then we need to show that 
$(pr_{j+1}V_{2n-j,j-1})^\perp\cap W_{2n-j,j}$ is not zero. But the first space is $(2n-j-1)$-dimensional and
the second one is $2(2n-j)$-dimensional and the restriction of the symplectic form to $W_{2n-j,j}$ is non-degenerate.
Now assume that we have fixed $V_{i+1,j}$. Using the same arguments as above one can show that 
there exists $V_{i,j}$ satisfying
\[
pr_j V_{i,j-1} \subset V_{i,j}\subset V_{i+1,j}\cap W_{i,j}, \ pr_{j+1}\dots pr_{2n-i}V_{i,j} 
\text{ is isotropic}.
\] 
\end{proof}

\begin{cor}\label{irred}
The variety $\SF^a_{2n}$ is irreducible.
\end{cor}

We now establish a connection between the $SL_{2n}$ resolution $R_{2n}$ and its symplectic analogue
$SpR_{2n}$. Define an involution $\sigma:R_{2n}\to R_{2n}$ by the following formula: let
$\bV=(V_{i,j})_{1\le i\le j}$ be a point in $R_{2n}$. Then
\begin{equation}\label{barsigma}
(\sigma\bV)_{i,j}=V_{2n-j,2n-i}^\perp\cap W_{i,j}.
\end{equation}
\begin{prop}
Formula \eqref{barsigma} defines an order two automorphism of $R_{2n}$. The variety of $\sigma$
fixed points $R_{2n}^{\sigma}$ is isomorphic to $SpR_{2n}$. For the resolution map
$\pi_{2n}:SpR_{2n}\to \SF^a_{2n}$ one has $\pi_{2n}\sigma=\sigma\pi_{2n}$ (i.e. $\pi_{2n}$ is
$\sigma$-equivariant).
\end{prop}
\begin{proof}
First, we need to show that $\dim (V_{2n-j,2n-i}^\perp\cap W_{i,j})=i$.
We consider the case $i+j>2n$ (the opposite case is very similar). We note that
$\dim V_{2n-j,2n-i}=2n-j$ and thus $\dim V_{i,j}^\perp=j$. Since $V_{2n-j,2n-i}\subset W_{2n-j,2n-i}$ we have
\[
V^\perp_{2n-j,2n-i}\supset W^\perp_{2n-j,2n-i}\supset\{w_{i+1},\dots,w_j\}.
\]
Since $W_{i,j}=\spa(w_{i+1},\dots,w_j)$, we arrive at $\dim (V_{2n-j,2n-i}^\perp\cap W_{i,j})=j-(j-i)=i$.

Second, we need to prove that all other conditions from the definition of $R_{2n}$ are satisfied.
This is a direct verification.
\end{proof}

Recall that for a resolution of singularities $\pi:Y\to X$ 
an irreducible divisor $Z\subset Y$ is called
exceptional if $\dim\pi(Z)<\dim Z$.
\begin{prop}\label{except}
A divisor $Z_{i,j}$ is exceptional if and only if $j\ge n$ and $i+j<2n$.
\end{prop}
\begin{proof}
Fix a pair $i<j$ such that $j\ge n$ and $i+j<2n$.
We  prove that the divisor $Z_{i,j}$ is exceptional. Fix a point $(V_i)_{i=1}^n\in\SF^a_{2n}$ and
a point $\bV\in Z_{i,j}$ such that $\pi_{2n} \bV=(V_i)_{i=1}^n$, i.e. $V_{i,i}=V_i$. 
We prove that the preimage $\pi_{2n}^{-1} (V_i)_{i=1}^n$ is at least one-dimensional.
Let us construct a one-dimensional family of elements $\bU\in Z_{i,j}$ such that $\pi_{2n}\bU=\pi_{2n}\bV$.
We set $U_{k,l}=V_{k,l}$ if $l\le j$. Now (as in the proof of Lemma \ref{surj}) let us try to 
extend already fixed components of $\bU$ to some element of $Z_{i,j}$. 
We start with $U_{2n-j-1,j+1}$. In order to guarantee that $\bU\in SpR_{2n}$, we need the following 
conditions:
\begin{equation}\label{U}
pr_{j+1} U_{2n-j-1,j}\subset U_{2n-j-1,j+1}\subset W_{2n-j-1,j+1},\ U_{2n-j-1,j+1} \text{ is isotropic}.
\end{equation}
Since $V_{i,j}\supset w_{j+1}$ ($\bV\in Z_{i,j}$) we have $\dim pr_{j+1} U_{2n-j-1,j}=2n-j-2$.
The restriction of the symplectic form to $W_{2n-j-1,j+1}$ is non-degenerate, thus 
the set of solutions of \eqref{U} is isomorphic to $\bP^1$. 
Now using the same arguments as in Lemma \ref{surj} we define successively 
$U_{2n-j-2,j+1},\dots,U_{i,j+1}$. Now let us try to define $U_{i-1,j+1}$. 
We want $U_{i-1,j+1}$ to be contained in the (already defined) $U_{i,j+1}$. We note that
$U_{i,j}=V_{i,j}=V_{i-1,j+1}\oplus \bC w_{j+1}$ and $pr_{j+1} U_{i,j}\subset U_{i,j+1}$ (by induction).
Therefore, we can set $U_{i-1,j+1}=V_{i-1,j+1}$, which leads to $U_{i,j}=V_{i,j}=U_{i-1,j+1}\oplus \bC w_{j+1}$. 
We note that since we want $\bU\in Z_{i,j}$ such a relation has to hold. 
Now according to the proof 
of Lemma \ref{surj}, we can define the remaining $U_{k,l}$ (by increasing induction on $l$ and decreasing
induction on $k$) in such a way that $\bU\in SpR_{2n}$. Moreover, by the construction as above, 
$\bU\in Z_{i,j}$. 

Now let us fix $j<n$. We prove that the divisor $Z_{i,j}$ with $i\le j$ is non-exceptional. 
More precisely, we prove that
$\pi_{2n}$ restricted to $Z_{i,j}^\circ$ (see \eqref{circ}) is one-to-one. So let $\bV\in Z_{i,j}^\circ$.

First, we claim that if $w_{l+1}\in V_{k,l}$ then $l=j$ and $i\le k\le j$. In fact, assume 
that for a given $l$ the number $k$ is minimal with the property $w_{l+1}\in V_{k,l}$. Then
$k\ne 1$ (unless $i=1$) since $w_{l+1}\in V_{1,l}$ means $\bV\in Z_{1,l}$. We know that
$w_{l+1}\notin V_{k-1,l}$. Hence $pr_{l+1} V_{k-1,l}=V_{k-1.l+1}$ and $V_{k,l}=V_{k-1,l}\oplus\bC w_{l+1}$.
Therefore, $V_{k,l}=V_{k-1,l+1}\oplus \bC w_{l+1}$, implying $\bV\in Z_{k,l}$. Since $\bV\in Z_{i,j}^\circ$,
we arrive at $(k,l)=(i,j)$. 

Second, we note that the above claim implies that
for any pair $(k,l)$ except for $l=j$, $i\le k\le j$ we have
\begin{equation}\label{pr}
\dim pr_{l+1} V_{k,l}=V_{k,l+1}.
\end{equation}
In fact, for such pairs $(k,l)$ the left and right hand sides have the same dimensions and the left hand
side is contained in the right hand side. From equation \eqref{pr} we obtain that
if a point $(V_i)_{i=1}^n\in \SF^a_{2n}$ is fixed, the spaces $V_{k,l}$ with $k<i$ or $k>j$ or
$l\le j$ are fixed as well. Also, if one proves that the spaces $V_{k,j+1}$ are also uniquely fixed
by the choice of $(V_i)_{i=1}^n$, this would imply that the whole collection $\bV\in Z_{i,j}^\circ$ with
the property $\pi_{2n} \bV=(V_i)_{i=1}^n$ is unique.

Third, we note that $\bV\in Z_{i,j}^\circ$ implies that for $l>j$ one has 
$p_{1,\dots,k} V_{k,l}\ne 0$ (where $p_{1,\dots,l}$ 
is the corresponding Pl\"ucker coordinate). In fact, first $V_{1,2n-1}$ is a one-dimensional space,
not equal to $\bC w_{2n}$ (since $\bV\in Z_{i,j}^\circ$). Therefore, $p_1(V_{1,2n-1})\ne 0$.
Now the statement can be verified directly by decreasing induction on $l$ and increasing induction on $k$.
      
Finally, we are left to show that the values of $V_i$ and of $V_{k,l}$ with $l\le j$ fix $V_{k,j+1}$. 
First, $V_{j+1,j+1}=V_{j+1}$ is fixed. Now, we use the decreasing induction on $k$. 
The space $V_{k,j+1}$ has to satisfy
\[
V_{k,j+1}\subset V_{k+1,j+1}\cap W_{k,j+1}.
\]
We claim that the right hand side is (at most) $k$-dimensional. In fact, we have
$p_{1,\dots,k+1}(V_{k+1,j+1})\ne 0$.
Since $w_{k+1}\notin W_{k,j+1}$, the $\dim (V_{k+1,j+1}\cap W_{k,j+1})\le k$. Since 
$\dim V_{k,j+1}=k$, the space $V_{k,j+1}$ is uniquely fixed.

The last statement of the proposition is that the divisors $Z_{i,2n-i}$ are non-exceptional. This can be
proved along the same lines as the above case.       
\end{proof}

\subsection{Parabolic case.} 
We define a desingularization $SpR_\bd$ of $\SF^a_\bd$ in the following way.
Let $P_\bd$ be the subset of the set of positive roots for $\msp_{2n}$ corresponding to the
radical of the parabolic subalgebra defined by the roots $\al_{d_1},\dots,\al_{d_k}$. Explicitly,
$\al\in P_\bd$ if and only if there exists an $i=1,\dots,k$ such that $(\al,\omega_{d_i})>0$.
We sometimes write $(i,j)\in P_\bd$ instead of $\al_{i,j}\in P_\bd$.
Now, a point in $SpR_\bd$ is a collection of subspaces
$V_{i,j}\subset W$, $(i,j)\in P_\bd$ subject to the conditions
\begin{gather*}
\dim V_{i,j}=i,\ V_{i,j}\subset W_{i,j},\\
V_{i,j}\subset V_{i+1,j},\ pr_{j+1}V_{i,j}\subset V_{i,j+1},\\
V_{i,2n-i} \text{ are isotropic for } i=1,\dots,n.
\end{gather*}
We note that there is a natural embedding $SpR_\bd\subset \prod_{(i,j)\in P_\bd} Gr_i(W_{i,j})$.

\begin{prop}
$SpR_\bd$ is a tower of successive $\bP^1$-fibrations.
\end{prop}
\begin{proof}
The proof is very similar to the proof of Lemma \ref{P1} and we omit it.
\end{proof}

\begin{prop}
A map $(V_{i,j})_{(i,j)\in P_\bd}\mapsto (V_{d_i,d_i})_{i=1}^k$ is a birational surjective map
from $SpR_\bd$ to $\SF^a_\bd$.
\end{prop}
\begin{proof}
The proof is very similar to the proof of Lemma \ref{surj} and we omit it.
\end{proof}

\begin{cor}\label{parirred}
The varieties $\SF^a_\bd$ are irreducible.
\end{cor}

In what follows we denote the desingularization map $SpR_\bd\to\SF^a_\bd$ by $\pi_\bd$.
As in the case of the complete flag varieties, for any pair $(i,j)\in P_\bd$, we define the corresponding
divisor $Z_{i,j}\subset SpR_\bd$.

\begin{prop}\label{nonex}
The non-exceptional divisors are exactly the $Z_{i,j}$ with $(i,j)$ from the following list:
\begin{gather*}
(1,d_i-1), i=2,\dots,k;\qquad (d_i+1,d_j-1), i\le j-2, j\le k;\\
 (1,2n-1);\qquad (d_i+1,2n-d_i-1), i=1,\dots,k-1.
\end{gather*}
\end{prop}
\begin{proof}
The proof is similar to the proof of Proposition \ref{except} and we omit it.
\end{proof}

\section{Normal locally complete intersections}\label{Normal_l.c.i.}
Our next goal is to prove that all symplectic degenerate flag varieties are normal locally
complete intersections (and thus Cohen-Macaulay and Gorenstein).
We start with the case of the complete flags.
\subsection{Complete flag varieties.}
As in \cite{FF}, we first define an affine scheme $Q_{2n}$.
Let $W_1, \dots, W_n$ be a collection of vector spaces with $\dim W_i=i$.
Also recall the space $W_{2n}$ with a basis $w_1,\dots,w_{2n}$, a non-degenerate symplectic form
and the projections $pr_k$ along $w_k$.
We now construct an affine scheme $Q_{2n}$ as follows. A point of $Q_{2n}$ is
a collection of linear maps
$$A_i: W_i\to W_{2n},\ i=1,\dots,n,\qquad  B_j: W_j\to W_{j+1},\ j=1,\dots,n-1$$
subject to the relations 
\begin{equation}\label{cd}
A_{i+1}B_i=pr_{i+1}A_i,\ i=1,\dots,n-1,
\end{equation}
and such that the image $A_n(W_n)$ is isotropic.
The following picture illustrates  the construction:
$$
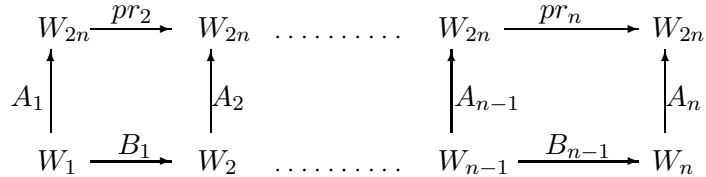
\begin{picture}(250,80)
\put(10,0){$W_1$}
\put(30,3){\vector(1,0){30}}
\put(70,0){$W_2$}
\multiput(100,0)(5,0){10}{\circle*{1}}
\put(160,0){$W_{n-1}$}
\put(190,3){\vector(1,0){45}}
\put(240,0){$W_{n}$}

\put(10,50){$W_{2n}$}
\put(70,50){$W_{2n}$}
\put(160,50){$W_{2n}$}
\put(240,50){$W_{2n}$}
\put(30,53){\vector(1,0){30}}
\put(185,53){\vector(1,0){50}}
\multiput(100,50)(5,0){10}{\circle*{1}}

\put(0,24){$A_1$}
\put(75,24){$A_2$}
\put(166,24){$A_{n-1}$}
\put(245,24){$A_{n}$}
\put(40,6){$B_1$}
\put(200,6){$B_{n-1}$}

\put(15,14){\vector(0,1){31}}
\put(75,14){\vector(0,1){31}}
\put(165,14){\vector(0,1){31}}
\put(245,14){\vector(0,1){31}}

\put(38,57){$pr_2$}
\put(198,57){$pr_{n}$}
\end{picture}
$$

We also consider an open part $Q_{2n}^\circ\subset Q_{2n}$ consisting of collections $(A_i, B_j)$
such that  $\ker A_i=0$ for all $i$.  
The group $\Gamma=\prod_{i=1}^{n} GL(W_i)$ acts freely on $Q^\circ_{2n}$
via the change of bases. Consider the map 
\[
Q^\circ_{2n}\to \SF^a_{2n},\quad (A_i,B_j)\mapsto (\mathrm{Im} A_1,\dots ,\mathrm{Im} A_{n}). 
\]

\begin{lem}
The map $Q^\circ_{2n}\to \SF^a_{2n}$ is a locally trivial $\Gamma$-torsor in the Zariski topology.
The dimension of $Q^\circ_{2n}$ 
is equal to $n^2 + 1^2+2^2+\ldots+n^2$.
\end{lem}
\begin{proof}
Consider the embedding $\SF^a_{2n}\hk \prod_{d=1}^{n} SpGr_d(2n)$. 
For a point $p\in \SF^a_{2n}$ let $U\ni p$ be an open part of $\prod_{d=1}^{n} SpGr_d(2n)$
such that all tautological bundles on Grassmannians are trivial on $U$. Let 
$U'=U\cap \SF^a_{2n}$. Then on $U'$ the map $Q^\circ_{2n}\to \SF^a_{2n}$ has a section.
Now using the $\Gamma$ action on $Q_{2n}$ we obtain that $Q^\circ_{2n}\to \SF^a_{2n}$ is 
$\Gamma$-torsor. In particular, $\dim Q^\circ_{2n}=\dim \SF^a_{2n} + \dim \Gamma$.
\end{proof}

We note that $Q_{2n}$ is a subscheme in the affine
space
\begin{equation}\label{affine}
\prod_{i=1}^{n} Hom(W_i,W_{2n})\times \prod_{i=1}^{n-1} Hom(W_i,W_{i+1}).
\end{equation}

\begin{lem}
$Q^\circ_{2n}$ is a locally complete intersection.
\end{lem}
\begin{proof}
We note that the condition $A_{i+1}B_i=pr_{i+1}A_i$ produces $2n\times i$ equations
(the number of equations is equal to $\dim Hom(W_i,W_{2n})$). Also the condition that
$A_n(W_n)$ is Lagrangian produces another $n(n-1)/2$ equations.
Now our lemma follows from the equality
$$
\dim Q^\circ_{2n}= \sum_{i=1}^{n} 2ni + \sum_{i=1}^{n-1} i(i+1) -\sum_{i=1}^{n-1} 2ni - \frac{n(n-1)}{2}.
$$
\end{proof}

\begin{thm}\label{lci}
The degenerate flag varieties $\SF^a_{2n}$ are normal locally complete
intersections (in particular, Cohen-Macaulay and even Gorenstein).
\end{thm}
\begin{proof}
Since $Q^\circ_{2n}\to \SF^a_{2n}$ is a torsor, it suffices to prove that 
$Q^\circ_{2n}$ is a normal reduced scheme (i.e. a variety).
Since $Q^\circ_{2n}$ is a locally complete intersection, the property of being reduced
(resp. normality) of $Q^\circ_{2n}$ follows from the fact
that the singularities of $Q^\circ_{2n}$ are contained in the subvariety of codimension
at least two by the virtue of~Proposition~5.8.5 (resp.~Theorem~5.8.6)
of~\cite{ega}. Using again that $Q^\circ_{2n}\to \SF^a_{2n}$ is a torsor,
it suffices to prove that
the codimension of the variety of singular points of $\SF^a_{2n}$ is at least two.
We prove this statement in a separate lemma.
\end{proof}

\begin{lem}
$\SF^a_{2n}$ is smooth off codimension two.
\end{lem}
\begin{proof}
We use the desingularization $\pi_{2n}:SpR_{2n}\to \SF^a_{2n}$. 
Since $SpR_{2n}$ is smooth, it
suffices to show that the map $\pi_{2n}$ is one-to-one on $\pi_{2n}^{-1} M$, where
$M\subset \SF^a_{2n}$ is an open part such that the codimension of the complement of $M$ is at least two.
We set $M$ to be the union of $\pi_{2n} SpR_{2n}^\circ$ ($U$ is the open cell in $SpR_{2n}$) and 
$\pi_{2n} Z^\circ_{i,j}$ 
for all non-exceptional divisors $Z_{i,j}$. 
\end{proof}

The following theorem can be proved along the same lines as Theorem \ref{lci}.
\begin{thm}
The degenerate flag varieties $\SF^a_\bd$ are normal locally complete
intersections (in particular, Cohen-Macaulay and Gorenstein).
\end{thm}

\section{The singularities of $\SF^a_\bd$ are rational.}\label{Singularities}
\subsection{The complete flag varieties}
In this section we prove that the singularities of $\SF^a_{2n}$ are terminal, canonical and rational. 
We use the sections $s_{i,j}$ as above in order to compute the line bundle
$\omega_{SpR_{2n}}\T\pi_{2n}^* \omega_{\SF^a_{2n}}^{-1}$. 

We use the notation $\om_{i,j}$ for the determinant of the $i$-dimensional bundle
on $SpR_{2n}$, whose fiber at a point $\bV$ equals $\Lambda^i(V_{i,j})^*$.
Also, we denote by
$\om_{i}$ the line bundle on $\SF^a_{2n}$, whose fiber at a point $(V_1,\dots,V_n)$ is equal to
$\Lambda^i(V_i^*)$.

Consider the following general situation.
Let $\rho:E\to B$ be a $\bP^1$-fibration with a section $s:B\to E$. Let $Z=s(B)\subset E$ be the corresponding
divisor.
Then for any line bundle $\cF$ on $E$ such that the restriction of $\cF$ to a fiber of $\rho$
is isomorphic to $\cO(k)$ one has
\begin{equation}\label{F}
\cF=\cO(kZ)\T \rho^*(\cF|_Z\T \cO(-kZ)|_Z).
\end{equation}

Assume that there exists a two-dimensional
bundle $\cL_2$ on $B$ and a line subbundle $\cL_1\subset \cL_2$ on $B$ such that $E=\bP(\cL_2)$ and
$Z=\bP(\cL_1)$. Let $\om$ be a line bundle such that the restriction of $\om$ to a fiber of $\rho$
is isomorphic to $\cO(1)$.
\begin{lem}\label{Kom}
$\omega^{-1}_E=\cO(Z)\T \om\T \rho^*(\om^{-1}|_Z\T \omega_B^{-1}).$
\end{lem}
\begin{proof}
First, \eqref{F} with $\cF=\om$ gives
\begin{equation}\label{omF}
\om=\cO(Z)\T\rho^*(\om|_Z\T \cO(-Z)|_Z).
\end{equation}
Second, \eqref{F} with $\cF=\omega_E^{-1}$ gives
\begin{equation}\label{omega-1}
\omega^{-1}_E=\cO(2Z)\T\rho^*(\omega_E^{-1}|_Z\T \cO(-2Z)|_Z).
\end{equation}
We note that 
\[
\rho^*(\om^{-1}|_Z\T \cO(-2Z)|_Z)=\rho^*(\omega_B^{-1}|_Z\T \cO(-Z)|_Z).
\]
Therefore, combining \eqref{omF} and \eqref{omega-1}, we arrive at
\[
\omega^{-1}_E=\cO(Z)\T \om\T \rho^*(\om^{-1}|_Z\T \omega_B^{-1}).
\]
\end{proof}

We now consider the $\bP^1$-fibration $\rho_l:SpR_{2n}(l)\to SpR_{2n}(l-1)$ and the corresponding
divisors $Z_{i,j}$.
\begin{lem}\label{Zom}
\[
\cO(Z_{i,j})=
\begin{cases}
\om_{1,j}\T \om_{1,j+1}^*,\text{ if } i=1,\\
\om_{i,j}\T \om_{i-1,j}^*\T \om_{i,j+1}^*\T \om_{i-1.j+1}, \text{ if } i>1 \text{ and } i+j<2n,\\
\om_{i,j}\T (\om_{i-1,j}^*)^{\T 2}\T \om_{i-1.j+1}, \text{ if } i>1 \text{ and } i+j=2n.
\end{cases}
\]
\end{lem}
\begin{proof}
The first two cases are worked out in \cite{FF}. The only new case is the last one $i+j=2n$.
Consider the space $V_{i,2n-i}$. We know that
\[
V_{i-1,2n-i}\subset V_{i,2n-i}\subset V_{i-1,2n-i}^\perp\cap W_{i,2n-i}.
\]
Let $\om_{i,j}^\perp$ be the line bundle on $SpR_{2n}$, whose fiber at a point $\bV$ is equal
to the top wedge power of $(V_{i,j}^\perp)^*$. Then one shows that 
$\det \om_{i,j}^\perp=\det \om_{i,j}^*$. Therefore,
using \eqref{F} we arrive at
\[
\om_{i,2n-i}=\cO(Z_{i,2n-i})\T \om_{i-1,2n-i+1}^*\T \om_{i-1,2n-1}^{\T 2}.
\]
\end{proof}

\begin{thm}\label{term}
$\omega_{SpR_{2n}}=\bigotimes_{i=1}^n (\om_{i,i}^*)^{\T 2}\T \prod_{i+j<2n, j\ge n} \cO(Z_{i,j})$.
\end{thm}
\begin{proof}
From Lemma \ref{Kom} we obtain
\[
\omega^{-1}_{SpR_{2n}}=
\bigotimes_{i,j} \cO(Z_{i,j}) \T  \bigotimes_{i,j} \om_{i,j} \T
\bigotimes_{j-i>2} \om^*_{i,j}.
\]
Now using Lemma \ref{Zom} we are left to show that
\[
\bigotimes_{i,j} \cO(Z_{i,j}) \T \bigotimes_{i=1}^n \om_{i,i} 
\T \bigotimes_{i=1}^{n-1} \om_{i,i+1}=
\bigotimes_{i=1}^n \om_{i,i}^{\T 2}\T \bigotimes_{j\ge n, i+j<2n} \cO(-Z_{i,j}).
\]
This equality follows from Lemma \ref{Zom} again by writing both sides in terms of $\om_{i,j}$ only.
\end{proof}

\begin{cor}\label{KSF}
$\omega_{\SF^a_{2n}}=\bigotimes_{i=1}^n (\om_i^*)^{\T 2}$.
\end{cor}
\begin{proof}
Recall the open cell $SpR_{2n}^\circ\subset SpR_{2n}$ with the property that the restriction of 
$\pi_{2n}$ to $SpR_{2n}^\circ$
is one-to-one. Consider a subvariety $M\subset SpR_{2n}$ defined by
\[
M=SpR_{2n}^\circ\cup \bigcup_{j<n} Z^\circ_{i,j}\cup\bigcup_{i+j=2n} Z^\circ_{i,j}.
\]
Then the restriction of $\pi_{2n}$ to $SpR_{2n}\setminus M$ is one-to-one and 
$\pi_{2n}(SpR_{2n}\setminus M)$ has codimension at least two in $\SF^a_{2n}$.
Because of Theorem \ref{term}, the canonical line bundle $\omega$ on $M$ is
equal to $\bigotimes_{i=1}^n (\om_{i,i}^*)^{\T 2}$.
Since we also have $\pi_{2n}^*\omega_{\SF^a_{2n}}=\omega_{SpR_{2n}}$ on the image of $M$, we 
obtain $\omega_{\SF^a_{2n}}=\bigotimes_{i=1}^n (\om_i^*)^{\T 2}$.
\end{proof}

\begin{thm}
The singularities of $\SF^a_{2n}$ are terminal  and hence  rational.
\end{thm}
\begin{proof}\
Using Corollary \ref{KSF} and Theorem \ref{term} we obtain
\[
\omega_{SpR_{2n}}=\pi_{2n}^* \omega_{\SF^a_{2n}}\T \prod_{i+j<2n, j\ge n} \cO(Z_{i,j}).
\]
Thus the singularities are terminal (and hence canonical)  by Theorem \ref{term}
(all exceptional divisors appear in the discrepancy with coefficient one).
Now rationality follows from \cite{E}, Theorem 1. In fact, the varieties $\SF^a_{2n}$
are locally complete intersections and so Gorenstein. Now, by the projection formula and Theorem \ref{term},
one gets
\[
(\pi_{2n})_*\omega_{SpR_{2n}}=
\bigotimes_{i=1}^n (\om_{i,i}^*)^{\T 2}\T (\pi_{2n})_*\prod_{i+j<2n, j\ge n} \cO(Z_{i,j}).
\]
Since $\cO(Z_{i,j})\supset \cO$ and $\SF^a_{2n}$ is normal (thus $(\pi_{2n})_*\cO=\cO$), we obtain
\[
(\pi_{2n})_*\omega_{SpR_{2n}}\supset \omega_{\SF^a_{2n}}.
\]
Using the natural opposite inclusion, we arrive at the isomorphism.
\end{proof}

\subsection{Rationality for parabolic flag varieties.}
In this subsection we work out the case of partial degenerate flag varieties.
Our goal is to compute $\omega_{SpR_\bd}\T\pi_\bd^* \omega_{\SF^a_\bd}^{-1}$ 
(note that $\SF^a_\bd$ is singular but Gorenstein, thus
its dualizing complex is a line bundle). 
We will prove that the expression for $\omega_{\SF^a_\bd}$ coincides with 
the one in the classical situation.
So let $\SF_\bd$ be the classical parabolic flag variety.
\begin{lem}\label{K}
$\omega^{-1}_{\SF_\bd}=\om_{d_1}^{\T d_2}\T \prod_{i=2}^{k-1} \om_{d_i}^{\T (d_{i+1}-d_{i-1})}\T
\om_{d_k}^{\T (2n+1-d_k-d_{k-1})}.$
\end{lem}
\begin{proof}
Since the canonical line bundle is the top wedge power of the cotangent bundle, its weight is given by
$-2\rho_{Sp_{2n}} + 2\rho_{L_\bd}$, where $2\rho_{Sp_{2n}}$ is the sum of all positive roots of $\msp_{2n}$
and $2\rho_{L_\bd}$ is the sum of all positive roots of the Levi subalgebra corresponding to $\bd$.
Now direct computation gives the desired formula.
\end{proof}

Now our strategy is as follows. We define a line bundle $\tilde \omega$ on $\SF^a_\bd$
by the formula
\[
\tilde \omega^{-1}=\om_{d_1}^{\T d_2}\T \prod_{i=2}^{k-1} \om_{d_i}^{\T (d_{i+1}-d_{i-1})}\T
\om_{d_k}^{\T (2n+1-d_k-d_{k-1})}.
\]
We will show that the discrepancy  $\omega_{SpR_\bd}\T\pi_\bd^* (\tilde \omega^{-1})$
is given by the product of exceptional divisors with positive coefficients. Since the images
of the exceptional divisors have codimension at least two, this will prove that in the degenerate case
the expression for the canonical class coincides with the one in the classical situation.

We start with the following lemma.
\begin{lem}\label{Oom}
We have the equality
\[
\omega^{-1}_{SpR_\bd}=\bigotimes_{(i,j)\in P_\bd}\cO(Z_{i,j})\T \bigotimes_{(i,j)\in B_\bd} \om_{i,j},
\]\
where $B_\bd\subset P_\bd$ is the subset of pairs from the following list:
\begin{gather*}
(1,d_1), (2,d_1),\dots,(d_1,d_1),\\
(d_1,d_1+1), (d_1,d_1+2), \dots, (d_1,d_2),\\
(d_1+1,d_2), (d_1+2,d_2),\dots,(d_2,d_2),\\
\dots\\
(d_{k-1}+1,d_k), (d_{k-1}+2,d_k),\dots,(d_k,d_k),\\
(d_k,d_k+1),(d_k,d_k+2),\dots, (d_k,2n-d_k).
\end{gather*}
\end{lem}
\begin{proof}
Follows from Lemma \ref{Kom}. In fact the sections $s_{i,j}$ are constructed in such
a way that the space $V_{i,j}$ at a point of the section is given by $V_{i-1,j+1}$ plus
some constant vector. The set $B_\bd$ is exactly the set of pairs $(i_0,j_0)$ which can not be witten
as $i_0=i-1$, $j_0=j+1$ for some $(i,j)\in P_\bd$.
\end{proof}

We now prove the main theorem of this section.
\begin{thm}
There exist non-negative integers $a_{i,j}$, $(i,j)\in P_\bd$ such that
\[
\omega_{SpR_\bd}\T\pi_\bd^*(\tilde \omega^{-1})=\bigotimes_{(i,j)\in P_\bd}\cO(a_{i,j}Z_{i,j}).
\]
In addition, $a_{i,j}=0$ if and only if $Z_{i,j}$ is not exceptional.
\end{thm}
\begin{proof}
We use the formula from Lemma \ref{Zom}:
\begin{equation}\label{expr}
\cO(Z_{i,j})=
\begin{cases}
\om_{1,j}\T \om_{1,j+1}^*,\text{ if } i=1,\\
\om_{i,j}\T \om_{i-1,j}^*\T \om_{i,j+1}^*\T \om_{i-1.j+1}, \text{ if } i>1 \text{ and } i+j<2n,\\
\om_{i,j}\T (\om_{i-1,j}^*)^{\T 2}\T \om_{i-1.j+1}, \text{ if } i>1 \text{ and } i+j=2n.
\end{cases}
\end{equation}
Using Lemma \ref{Oom} we obtain
\[
\omega_{SpR_\bd}\T \pi_\bd^*(\tilde \omega^{-1})=
\bigotimes_{(i,j)\in P_\bd}\cO(-Z_{i,j})\T \bigotimes_{(i,j)\in B_\bd} \om_{i,j}^{-1}\T
\pi_\bd^*(\tilde \omega^{-1}).
\]
Our goal is to find strictly positive integers  $b_{i,j}=a_{i,j}+1$ such that
\begin{multline}\label{bij}
\bigotimes_{(i,j)\in B_\bd} \om_{i,j}^{-1}\T
\om_{d_1}^{\T d_2}\T \prod_{i=2}^{k-1} \om_{d_i}^{\T (d_{i+1}-d_{i-1})}\T
\om_{d_k}^{\T (2n+1-d_k-d_{k-1})}\\
 = \bigotimes_{(i,j)\in P_\bd}\cO(b_{i,j}Z_{i,j}).
\end{multline}
In addition, we want $b_{i,j}=1$ if and only if $Z_{i,j}$ is non-exceptional.

First, formula \eqref{bij} implies $b_{d_1,d_1}=d_2-1$, $b_{d_i,d_i}=d_{i+1}-d_{i-1}-1$ for $i=2,\dots,k-1$
and $b_{d_k,d_k}=2n-d_k-d_{k-1}$. Now formula \eqref{expr} gives unique way to determine other
$b_{i,j}$ in the following order:
\begin{gather*}
b_{d_1,d_1}, b_{d_1-1,d_1},\dots, b_{1,d_1},\\
b_{d_1,d_1+1}, b_{d_1-1,d_1+1},\dots, b_{1,d_1+1},\\
\dots\\
b_{d_1,d_2-1}, b_{d_1-1,d_2-1},\dots, b_{1,d_2-1},\\
b_{d_2,d_2}, b_{d_2-1,d_2},\dots, b_{1,d_2},\\
\dots\\
b_{d_2,d_3-1}, b_{d_2-1,d_3-1},\dots, b_{1,d_3-1},\\
\dots\\
\dots\\
b_{d_k,d_k}, b_{d_k-1,d_k},\dots, b_{1,d_k},\\
b_{d_k,d_k+1}, b_{d_k-1,d_k+1},\dots, b_{1,d_k+1},\\
\dots\\
b_{d_k,2n-d_k}, b_{d_k-1,2n-d_k},\dots, b_{1,2n-d_k},\\
b_{d_k-1,2n-d_k+1}, b_{d_k-2,2n-d_k+1},\dots, b_{1,2n-d_k+1},\\\
\dots\\
b_{2,2n-2}, b_{1,2n-2},\\
b_{1,2n-1}.
\end{gather*}
We now write down the values of the solutions $b_{i,j}$ (we assume $d_0=0$):

\smallskip
\noindent
Let $d_1\le j\le d_2-1$ and $1\le i\le d_1$. Then $$b_{i,j}=d_2-j+i-1.$$
\noindent
Let $d_{s-1}\le j\le d_s-1$ and $d_l+1\le i\le d_{l+1}$ for $0\le l<s<k$. Then
$$b_{i,j}=d_s-d_l-j+i-1.$$
\noindent
Let $d_k\le j< 2n-d_k$ and $d_l+1\le i\le d_{l+1}$ for $0\le l <k$. Then
$$b_{i,j}=2n-d_k-d_l-j+i.$$
\noindent
Let $2n-d_s\le j< 2n-d_{s-1}-1$ and $i=2n-j$ for $1\le s \le k$. Then
$$b_{i,j}=2n-j-d_{s-1}.$$
\noindent
Let $2n-d_s\le j\le  2n-d_{s-1}-1$ and $d_s-1\ge i\ge d_{s-1}+1$ for $1\le s \le k$. Then
$$b_{i,j}=2n-2d_{s-1}-j+i.$$
\noindent
Let $2n-d_s\le j\le  2n-d_{s-1}-1$ and $d_l\ge i\ge d_{l-1}+1$ for $1\le s \le k$, $s-1\ge l\ge 1$. Then
$$b_{i,j}=2n-d_{s-1}-d_{l-1}-j+i.$$

In particular, one sees that $b_{i,j}\ge 1$ and $b_{i,j}=1$ if and only if $(i,j)$ is from the
following list
\begin{gather*}
(1,d_i-1), i=2,\dots,k;\qquad (d_i+1,d_j-1), i\le j-2, j\le k;\\
 (1,2n-1);\qquad (d_i+1,2n-d_i-1), i=1,\dots,k-1.
\end{gather*}
Now Proposition \ref{nonex} gives the desired result.
\end{proof}

\begin{cor}
$\omega_{\SF^a_\bd}=\tilde \omega$.
\end{cor}
\begin{proof}
Similar to the proof of Corollary \ref{KSF}.
\end{proof}

\begin{cor}
$\SF^a_\bd$ has terminal and thus canonical and rational singularities.
\end{cor}

\section{Frobenius splitting, the BWB-type theorem and graded character formula}\label{Applications}
In this section we derive several applications of the results in the previous sections. 
\subsection{Frobenius splitting.}
The varieties $\SF^a_\bd$ and $SpR_\bd$ are  defined over any field ${\bf k}$
(simply replacing $\bc$ by ${\bf k}$ in all definitions).
\begin{thm}
The varieties $\SF^a_\bd$ and $SpR_\bd$ over $\overline{\mathbb F}_p$ 
are Frobenius split for all prime $p$.
\end{thm}
\begin{proof}
Lemma \ref{Oom} says that
\[
\omega^{-1}_{SpR_\bd}=\bigotimes_{(i,j)\in P_\bd}\cO(Z_{i,j})\T \bigotimes_{(i,j)\in B_\bd} \om_{i,j}.
\]
Now since $\bigotimes_{(i,j)\in B_\bd} \om_{i,j}$ is base point free, the Frobenius splitting 
for $SpR_\bd$ and $\SF^a_\bd$ follows
from the Mehta-Ramanathan criterion (Proposition $8$ of \cite{MR}).
\end{proof}

\subsection{The BWB-type theorem}
Let $\cL_\la$ be a line bundle on $\SF^a_{2n}$ which is the pull back
of the line bundle $\cO(1)$ on $\bP(V_\la^a)$ for a dominant $\msp_{2n}$-weight $\la$.
We prove the analogue of the Borel-Weil-Bott theorem.
\begin{thm}
We have
\begin{gather*}
H^0(\SF^a_{2n},\cL_\la)^*\simeq H^0(SpR_{2n},\pi_{2n}^*\cL_\la)^*\simeq V_\la^a,\\
H^{>0}(\SF^a_{2n},\cL_\la)=H^{>0}(SpR_{2n},\pi_{2n}^*\cL_\la)=0.
\end{gather*}
\end{thm}
\begin{proof}
First, we note that since $\SF^a_{2n}$ has rational singularities,
we have the equalities
\[\
H^k(\SF^a_{2n},\cL_\la)\simeq H^k(SpR_{2n},\pi_{2n}^*\cL_\la)
\]
for all $k\ge 0$.

Second, we prove that all non-zero cohomology $H^k(\SF^a_{2n},\cL_\la)$ vanish. In fact, first assume
$\la$ is regular. Then since the map  $\SF^a_{2n}\to \bP(V_\la^a)$ is an embedding, the line bundle
$\cL_\la$ is very ample. Therefore, for any $k$ and big enough $N$ one has
$H^k(\SF_{2n}^a,\cL_\la^{\T N})=0$. This implies $H^k(\SF_{2n}^a,\cL_\la)=0$, because
$\SF^a_{2n}$ is Frobenius split over $\overline{\mathbb F}_p$ for any $p$.

Now consider a non regular $\la$. Let $\SF^a_\bd$ be the corresponding
degenerate parabolic
flag variety, which is embedded into $\bP(V_\la^a)$. Then we have the following
commutative diagram of projections:
$$
\begin{picture}(100,110)

\put(0,10){$\SF_{2n}^a$}
\put(90,10){$\SF^a_\bd$}
\put(0,90){$SpR_{2n}$}
\put(90,90){$SpR_\bd$}

\put(28,13){\vector(1,0){60}}
\put(28,93){\vector(1,0){60}}
\put(15,87){\vector(0,-1){62}}
\put(98,87){\vector(0,-1){60}}

\put(55,14){$\phi$}
\put(55,95){$\eta$}
\put(0,56){$\pi_{2n}$}
\put(103,56){$\pi_\bd$}
\end{picture}
$$

Let $\cL'_\la$ be a line bundle on $\SF^a_\bd$  which is the pull back of the bundle $\cO(1)$
on $\bP(V_\la^a)$. Then $\cL_\la=\phi^*\cL'_\la$.
Since $\cL'_\la$ is very ample, and $\SF^a_\bd$ is Frobenius split over
$\overline{\mathbb F}_p$ for any $p$, $H^k(\SF^a_\bd,\cL'_\la)=0$
(for positive $k$). Since $\SF^a_\bd$ has rational singularities,
$H^k(SpR_\bd,\pi_\bd^*\cL'_\la)=H^k(\SF^a_\bd,\cL'_\la)(=0$ for positive $k$).
Now since $\eta$ is a fibration with the fibers being towers of successive
$\bP^1$-fibrations, we obtain $H^k(SpR_{2n},\eta^*\pi_\bd^*\cL'_\la)=
H^k(SpR_\bd,\pi_\bd^*\cL'_\la)(=0$ for positive $k$). Finally, since $\SF_{2n}^a$
has rational singularities,
and $\eta^*\pi_\bd^*\cL'_\la=\pi_{2n}^*\cL_\la$, we arrive at 
$$H^k(\SF_{2n}^a,\cL_\la)=
H^k(SpR_{2n},\pi_{2n}^*\cL_\la)=H^k(R_{2n},\eta^*\pi_\bd^*{\cL'}_\la),$$ 
which vanishes for $k>0$.

Third, we note that there exists an embedding $(V_\la^a)^*\hk H^0(\SF^a_{2n},\cL_\la)$.
In fact take an element $v\in (V_\la^a)^*\simeq H^0(\bP(V_\la^a),\cO(1))$.
Then restricting to the embedded variety $\SF^a_{2n}$ we obtain a section of $\cL_\la$. Assume that
it is zero. Then $v$ vanishes on the open cell $(N^-)^a\cdot \bC v_\la$. But
the linear span of the elements of this cell coincides with the whole representation $V_\la^a$.
Therefore, the restriction map  $(V_\la^a)^*\to H^0(\SF^a_{2n},\cL_\la)$ is an embedding.

Finally, we recall that the varieties $\SF_{2n}^a$ are flat degenerations of the classical flag
varieties. Since the higher cohomology of $\cL_\la$ vanish, we arrive at the equality of the dimensions
of $H^0(\SF_{2n}^a,\cL_\la)$ and of $V_\la$.
Therefore, the embedding  $(V_\la^a)^*\to H^0(\SF^a_{2n},\cL_\la)$ is an isomorphism.
\end{proof}

Similarly one proves a parabolic version of the BWB-type theorem:
\begin{thm}
Let $\la$ be a $\bd$-dominant weight, i.e. $(\la,\omega_d)>0$ implies $d\in\bd$.
Then there exists a map $\imath_\la:\SF^a_\bd\to\bP(V_\la^a)$.
We have
$$
H^0(\SF^a_\bd,\imath_\la^*\cO(1))^*\simeq V_\la^a,\
H^{>0}(\SF^a_\bd,\imath_\la^*\cO(1))=0.
$$
\end{thm}

\subsection{The $q$-character formula}
We now compute the $q$-character (PBW-graded character) of the modules $V_\la^a$
(for combinatorial formula see Theorem \ref{basis}).
For this we use the Atiyah-Bott-Lefschetz fixed points formula applied to the variety $SpR_{2n}$.
We first describe the fixed points explicitly.
\begin{lem}
The $T$-fixed points on $SpR_{2n}$ are labeled by the collections $\bs=(S_{i,j})$, $1\le i\le j< 2n$, 
$i+j\le 2n$, where $S_{i,j}$ are subsets of $\{1,\dots,2n\}$ satisfying the following properties:
\begin{itemize}
\item $S_{i,j}\subset \{1,\dots,i,j+1,\dots,n\}$,\  $\# S_{i,j}=i$,
\item $S_{i,j}\subset S_{i+1,j}\subset S_{i+1,j+1}\cup\{j+1\}$,
\item For any $i=1,\dots,n$ if $k\in S_{i,2n-i}$. then $2n+1-i\notin S_{i,2n-i}$. 
\end{itemize}
\end{lem}
\begin{proof}
Obviously, a collection $\bV\in SpR_{2n}$ is a $T$-fixed point if and only if each $V_{i,j}$ is
the linear span of some basis vectors $w_l$. Now each collection $\bS$ as above determines $\bV$ 
by the formula $V_{i,j}=\spa (w_l:\ l\in S_{i,j})$.
\end{proof}
We call a collection $\bS$ satisfying the conditions as above admissible. For an admissible
$\bS$ let $p(\bS)\in SpR_{2n}$ be the corresponding fixed point and let 
$p(S_{i,j})\in Gr_i(W_{i,j})$ be its $(i,j)$-th component.  

Recall the extended degenerate group $G^a\rtimes \bC^*$.
\begin{lem}
The action of the group $G^a$ and its extension $G^a\rtimes \bC^*$ 
on $\SF^a_{2n}$ lifts to $SpR_{2n}$. 
\end{lem}
\begin{proof}
Recall the embeddings $\SF^a_{2n}\hk \Fl^a_{2n}$ and $SpR_{2n}\hk R_{2n}$ (see Proposition \ref{sigma}).
Since the analogue of our Lemma for $SL_{2n}$ holds (see \cite{FF}), we obtain the desired
result for $Sp_{2n}$ as well.   
\end{proof}

In order to state the theorem we prepare some notations. Let $\bC[e^{\omega_1},\dots,e^{\omega_n},e^d]$
be the group algebra of the weight lattice of the extended Lie algebra $\g^a\oplus\bC d$. We sometimes
use the notations $z_i=e^{\omega_i}$, $q=e^d$. For an element $\mu=md+\sum_{i=1}^n m_i\omega_i$ 
we write $e^\mu=q^m\prod_{i=1}^n z_i^{m_i}$. Also for a homogeneous vector $v\in V_\la^a$ we denote 
by $\operatorname{wt}_q (v)$ the extended weight of $v$. 

Recall the Atiyah-Bott-Lefschetz formula (see \cite{AB}, \cite{T}): let $X$ be a smooth projective algebraic
$M$-dimensional variety and let $\cL$ be a line bundle on $X$. Let $T$ be an algebraic torus acting on $X$ with a finite
set $F$ of fixed points. Assume further that $\cL$ is $T$-equivariant.
Then for each $p\in F$ the fiber $\cL_p$ is $T$-stable. We note also that since $p\in F$, the tangent space
$T_pX$ carries a natural $T$-action. Let $\gamma_1^p,\dots,\gamma_M^p$ be the
weights of the eigenvectors of $T$-action on $T_pX$.
Then the Atiyah-Bott-Lefschetz formula gives the following expression
for the character of the Euler characteristics:
\begin{equation}\label{ABL}
\sum_{k\ge 0} (-1)^k \ch H^k(X,\cL)=\sum_{p\in F} \frac{\ch \cL_p}{\prod_{l=1}^M (1-e^{-\gamma_l^p})}.
\end{equation}
We apply this formula for $X=SpR_{2n}$, $\cL=\pi_{2n}^* \cL_\la$ with the action of the extended 
torus $T\cdot \bC^*$. 
Since $H^{>0}(R_n,\pi_n^*\cL_\la)=0$, the Euler
characteristics coincides with the character of the zeroth cohomology, i.e. with the character of
$(V_\la^a)^*$. Therefore, for each admissible $\bS$ we need to compute the character of 
$\pi_{2n}^* \cL_\la$ at $p(\bs)$ and the eigenvalues of the torus action in $T_{p(\bs)}SpR_{2n}$.

Let $\imath_\la:\SF^a_{2n}\to \bP(V_\la^a)$ be the standard map (which is an embedding for regular $\la$).
Then $\ch (\pi_{2n}^* \cL_\la)_{p(\bs)}=e^{-\operatorname{wt}_q(\imath_\la p(\bs))}$ 
(the minus sign comes from the fact that
a fiber of $\cO(1)$ is a dual line). We note that the weight of $\imath_\la p(\bs)$ depends only on
the diagonal entries $S_{i,i}$.

Now let us compute the eigenvalues of the tangent action of the torus at a point $p_\bs$. 
For each pair $(i,j)$, $1\le i\le j< 2n$, $i+j\le 2n$ define a collection $S'_{i,j}$ as follows. 

First, let $i+j<2n$. Given the sets $S_{i-1,j}$ and $S_{i,j+1}$,
let us look at the possible values of $S_{i,j}$ keeping $\bS$ admissible. 
We denote such a possible collections by $\bar S_{i,j}$ in order to distinguish it from the already
fixed component $S_{i,j}$.
The definition of 
admissibility says that there exist exactly two variants for $\bar S_{i,j}$,
namely
\[
\bar S_{i,j}=S_{i-1,j}\cup \{a\} \text{ or } \bar S_{i,j}=S_{i-1,j}\cup \{b\},
\]
where $\{a,b\}=S_{i,j+1}\cup\{j+1\}\setminus S_{i-1,j}$. Given a collection $\bS$ we denote
the numbers $a,b$ as above by $a^\bs_{i,j}$ and $b^\bs_{i,j}$. We have:
\[
S_{i,j}=S_{i-1,j}\cup \{a^\bs_{i,j}\},\ \ S_{i,j+1}\setminus S_{i-1,j}=\{a^\bs_{i,j},b^\bs_{i,j}\}.
\]
We denote by $S'_{i,j}$ the set $S_{i,j}\setminus \{a^\bs_{i,j}\}\cup \{b^\bs_{i,j}\}$.

Second, assume $i+j=2n$. Given the set $S_{i-1,j}$,
let us look at the possible values of $\bar S_{i,j}$ keeping $\bS$ admissible. The definition of 
admissibility says that there exist exactly two variants for $\bar S_{i,j}$,
namely
\[
\bar S_{i,j}=S_{i-1,j}\cup \{a\} \text{ or } \bar S_{i,j}=S_{i-1,j}\cup \{b\},
\]
where $\{a,b\}=\{1,\dots,i,2n-i+1,\dots,2n\}\setminus \{2n+1-l:\ l\in S_{i-1,j}\}$. 
We also denote the numbers $a,b$ by $a^\bs_{i,j}$ and $b^\bs_{i,j}$. We set
$S_{i,j}=S_{i-1,j}\cup \{a^\bs_{i,j}\}$ and we 
denote by $S'_{i,j}$ the set $S_{i,j}\setminus \{a^\bs_{i,j}\}\cup \{b^\bs_{i,j}\}$.

Recall that the variety $\SF^a_{2n}$ sits inside the product of Grassmann varieties
$\prod Gr_i(W_{i,j})$.
Each $\bigwedge^i(W_{i,j})$ is acted upon by $\g^a\oplus\bC d$
and therefore each
Grassmannian  carries a natural action of the group $G^a\rtimes\bC^* $ (the additional $\bC^*$ part
corresponds to the PBW-grading operator). 
Thus for each collection $S_{i,j}$ we have the corresponding weight $\operatorname{wt}_q p(\bs_{i,j})$, which is the
weight of the corresponding point in $\bigwedge^i(W_{i,j})$.

\begin{thm}
The $q$-character of the representation $V^a_\la$ is given by the sum over all
admissible collections $\bS$ of the summands
\begin{equation}\label{AB}
\frac{e^{\operatorname{wt}_q (\imath_\la p(\bs))}}
{\prod_{\substack{i+j\le 2n\\ 1\le i\le j<2n}} \left(1-e^{\operatorname{wt}_q p(S'_{i,j})-\operatorname{wt}_q {p(S_{i,j})}}\right)}.
\end{equation}
\end{thm}
\begin{proof}
Recall that $SpR_{2n}$ can be constructed as a tower of successive $\bP^1$-fibrations 
$SpR_{2n}(l)\to SpR_{2n}(l-1)$. 
Fix an admissible $\bS$. Then the surjections $SpR_{2n}\to SpR_{2n}(l)$ define the $T$-fixed points 
$p(\bS(l))$ in each $SpR_{2n}(l)$
(note that $\bS(l)$ consists of $S_{i,j}$ such that for $\beta_k=\al_{i,j}$ one has $k\le l$).
For each $l=1,\dots,M$ we denote by $v_l\in T_{p(\bS(l))}SpR_{2n}(l)$ a tangent vector to the fiber
of the map $SpR_{2n}(l)\to SpR_{2n}(l-1)$ at the point $p(\bS(l-1))$.
Then it is easy to see that the weights of the eigenvectors of the $T$ action in $T_{p(\bs)}SpR_{2n}$
are exactly the weights of the vectors $v_l$, $l=1,\dots,M$.

So let us fix $l$,  $1\le l\le M$ and $i,j$ with  $\al_{i,j}=\beta_l$.
Let us denote by $Y_l$ the set of all pairs $(k,m)$ such that for the root 
$\al_{k,m}=\beta_r$ one has $r\le l$.
Then the fiber $\bP^1$ of the map $SpR_{2n}(l)\to SpR_{2n}(l-1)$
at the point $p(\bS(l-1))$ consists of all collections $(V_{k,m})$ with $(k,m)\in Y_l$ subject to the following
conditions:
\begin{itemize}
\item $V_{k,m}=p(S_{k,m})$ if $\al_{k,m}\ne \beta_l$,\\
\item $V_{i,j}\supset p(S_{i-1,j})$,\\
\item $V_{i,j}\subset p(S_{i-1,j})\oplus\bC w_{a^\bS_{i,j}}\oplus\bC w_{b^\bS_{i,j}}$.
\end{itemize}
Now it is easy to see that the character of the tangent vector to this fiber at the point
$p(\bS(l-1))$ is equal to $e^{\operatorname{wt}_q p(S'_{i,j})-\operatorname{wt}_q {p(S_{i,j})}}$
(recall $a^\bS_{i,j}\in S_{i,j}$ and $S'_{i,j}=S_{i,j}\setminus \{a^\bS_{i,j}\} \cup \{b^\bS_{i,j}\}$).
\end{proof}

\begin{rem}
We note that the Euler characteristics 
$$\sum_{k\ge 0} (-1)^k \ch H^k(SpR_{2n},\pi_{2n}^*\cL_\la)$$
is equal to $\ch (V_\la^a)^*$. But in each summand \eqref{AB} both numerator and denominator
differ from the corresponding summand in the Atiyah-Bott-Lefschetz formula \eqref{ABL} by the change 
of variables $z_i\to z_i^{-1}$ and $q\to q^{-1}$. Via this change we pass from the character of
$(V_\la^a)^*$ to the character of $V_\la^a$. 
\end{rem}

\section*{Acknowledgments}
The main part of this paper was written during the E.~F. and P.~L. stay at the Hausdorff Research 
Institute for Mathematics.
The hospitality and perfect working conditions of the Institute are gratefully
acknowledged. We are grateful to E.~Amerik for the explanations about singularities and for pointing out 
the reference \cite{E}.
The work of E.~F. was partially supported
by the RFBR Grant 09-01-00058,
by the grant Scientific Schools 6501.2010.2 and by the Dynasty Foundation.
M.~F. was partially supported by the RFBR grant 09-01-00242, the HSE Science
Foundation award No. 11-09-0033,
the Ministry of Education and
Science of Russian Federation, grant No. 2010-1.3.1-111-017-029,
and the AG Laboratory HSE, RF government grant, ag. 11.G34.31.0023.

\end{document}